\documentclass[a4,reqno]{amsart}
\usepackage{amsmath,amssymb,mathrsfs}
\usepackage{graphicx,cite,times}
\usepackage{epstopdf}
\usepackage[colorlinks,
citecolor=red,
urlcolor=blue,
backref=page]{hyperref}

\usepackage{graphicx,subfigure}
\usepackage{bm}
\usepackage[all]{xy}

\usepackage{geometry}

\newtheorem{theo}{Theorem}[section]
\newtheorem*{theoa}{Theorem A}
\newtheorem*{theob}{Theorem B}
\newtheorem{lemm}[theo]{Lemma}
\newtheorem{rema}[theo]{Remark}
\numberwithin{equation}{section}

\allowdisplaybreaks

\begin{document}
\large
\title[The Quasi-Periodic Cauchy Problem for the Generalized Benjamin-Bona-Mahony Equation]{The Quasi-Periodic Cauchy Problem for the Generalized Benjamin-Bona-Mahony Equation on the Real Line}

\author{David Damanik}
\address{\scriptsize (D. Damanik)~Department of Mathematics, Rice University, 6100 S. Main Street, Houston, Texas
77005-1892}
\email{damanik@rice.edu}
\thanks{The first author (D. Damanik) was supported by Simons Fellowship $\# 669836$ and NSF grants DMS--1700131 and DMS--2054752}

\author{Yong Li}
\address{\scriptsize  (Y. Li)~Institute of Mathematics, Jilin University, Changchun 130012, P.R. China. School of Mathematics and Statistics, Center for Mathematics and Interdisciplinary Sciences, Northeast Normal University, Changchun, Jilin 130024, P.R.China.}
\email{liyong@jlu.edu.cn}
\thanks{The second author (Y. Li) was supported in part by NSFC grants 12071175, 11171132 and 11571065, National Basic Research Program of China Grant 2013CB834100, and Natural Science Foundation of Jilin Province 20200201253JC}

\author{Fei Xu}
\address{\scriptsize (F. Xu)~Institute of Mathematics, Jilin University, Changchun 130012, P.R. China.}
\email{stuxuf@outlook.com}
\thanks{The third author (F. Xu) was supported by Graduate Innovation Fund of Jilin University 101832018C162.
}


\begin{abstract}
This paper studies the existence and uniqueness problem for the generalized Benjamin-Bona-Mahony (gBBM) equation with quasi-periodic initial data on the real line. We obtain an existence and uniqueness result in the classical sense with arbitrary time horizon under the assumption of polynomially decaying initial Fourier data by using the combinatorial analysis method developed in earlier papers by Christ, Damanik-Goldstein, and the present authors. Our result is valid for exponentially decaying initial Fourier data and hence can be viewed as a Cauchy-Kovalevskaya theorem for the gBBM equation with quasi-periodic initial data. 
\end{abstract}

\maketitle

\tableofcontents

\newpage

\section{Introduction}

The equation
\begin{align}\label{standard.bbm}
u_t - u_{xxt} + u_x + u u_x = 0
\end{align}
was introduced by Benjamin, Bona, and Mahony in \cite{BBM72} as an improvement of the Korteweg–de Vries equation
\begin{align}\label{e.kdv}
u_t + u_{xxx} + uu_x = 0
\end{align}
for modeling unidirectional propagation of long waves of small amplitude. We will for simplicity refer to \eqref{standard.bbm} as BBM and to \eqref{e.kdv} as KdV.

Olver showed in \cite{O79} that BBM possesses exactly three independent and non-trivial conservation laws, whereas KdV is known to possess infinitely many \cite{MGK68}. Both equations admit solitary wave solutions. KdV can be described via a Lax pair \cite{L68}, whereas BBM cannot.

Establishing the existence of solutions to the Cauchy problem associated with BBM is simpler for decaying or periodic initial data. For some foundational results in these two special cases we refer the reader to \cite{BBM72} and \cite{MM77}, respectively.
In the present paper we are interested in studying spatially quasi-periodic solutions, which is a more challenging task.

For KdV with quasi-periodic initial data, the existence and uniqueness of solutions was studied by Tsugawa \cite{T12} and Damanik-Goldstein \cite{DG16JAMS}. More recently, the analogous problem for the generalized Korteweg–de Vries (gKdV) equation
\begin{align}\label{e.gkdv}
u_t + u_{xxx} + u^{p-1} u_x = 0
\end{align}
was studied by the three of us in \cite{DLX21ARXIV}. Via these works it is known that sufficiently small quasi-periodic initial data with exponentially decaying Fourier coefficients admit a local in time solution that remains quasi-periodic in the spatial variable with exponentially decaying Fourier coefficients. Indeed, within this class of functions, the solution is unique. For KdV one can go further and show that, for Diophantine frequency vector, the local result can be iterated in constant time steps. In this way, one obtains global existence and uniqueness \cite{DG16JAMS}. We mention in passing that the dependence on time is in this setting known to be almost periodic \cite{BDGL18}, which is a result in line with (and providing evidence for) the Deift conjecture \cite{D08, D17}, which states that the KdV equation with almost periodic initial data admits a global solution that is almost periodic both in space and time.

This passage from a local to a global result in \cite{DG16JAMS} rests on a rather involved spectral analysis of quasi-periodic Schr\"odinger operators \cite{DG14}. As an input of this kind is not available for gKdV, it is at present unclear how to leverage the local result from \cite{DLX21ARXIV} to a global result.

In this paper we want to discuss the existence and uniqueness of spatially quasi-periodic solutions of BBM, and in fact more generally of the generalized Benjamin-Bona-Mahony (gBBM) equation
\begin{align}\label{sl}
u_t-u_{xxt}+u_x+u^{p-1}u_x=0.
\end{align}

As initial data we consider quasi-periodic functions of the form
\begin{align}\label{dt}
u(0,x)=\sum_{n\in\mathbb Z^\nu}\hat u(n)e^{{\rm i}\langle n\rangle x},
\end{align}
where $2\leq p\in\mathbb N, x\in\mathbb R$ (the real line), $\omega=(\omega_1,\cdots,\omega_\nu)\in\mathbb R^\nu$ is a given wave vector, $n=(n_1,\cdots,n_\nu)\in\mathbb Z^\nu$ is the dual vector, and $\langle n\rangle\triangleq\langle n,\omega\rangle$ is the standard inner product defined by letting $\langle n,\omega\rangle:=\sum_{j=1}^\nu n_j\omega_j$.  As usual we assume that the wave vector is  non-resonant or rationally independent, that is, $\langle n\rangle=0$ implies that $n=0\in\mathbb Z^\nu$.

What we are interested in is the existence and uniqueness of spatially quasi-periodic solutions defined by the Fourier series
\begin{align}
u(t,x)=\sum_{n\in\mathbb Z^\nu}\hat u(t,n)e^{{\rm i}\langle n\rangle x}
\end{align}
to the quasi-periodic Cauchy problem \eqref{sl}--\eqref{dt} in the classical sense. Our main results are the following Theorem A (exponential decay) and Theorem B (polynomial decay) below.

\begin{theoa}
Suppose the Fourier coefficients $\hat u(n)$ of the initial data satisfy the following exponential decay condition,
\begin{align}\label{exd}
|\hat u(n)|\leq\mathcal A^{\frac{1}{p-1}}e^{-\rho|n|},\quad\forall n\in\mathbb Z^\nu,
\end{align}
where $\mathcal A>0$ and $0<\rho\leq1$.

Then the quasi-periodic Cauchy problem \eqref{sl}--\eqref{dt} has a unique spatially quasi-periodic solution in the classical sense on the time interval $[0,{\mathcal L}_p]$, where
\begin{align}\label{ga}
{\mathcal L}_p \triangleq \left(1-\frac{1}{p}\right)^{p-1}\frac{\rho^{(p-1)\nu}}{\mathcal A 6^{(p-1)\nu}}.
\end{align}
\end{theoa}

\begin{rema}\label{rem.1.1}
Theorem~A is a local existence and uniqueness result with arbitrary time horizon. That is, given any $T > 0$, we provide an explicit class of quasi-periodic initial data with exponential Fourier decay (namely those obeying \eqref{exd} with parameters $\mathcal A, \rho, \nu$ subject to the condition
$$
\left(1-\frac{1}{p}\right)^{p-1}\frac{\rho^{(p-1)\nu}}{\mathcal A 6^{(p-1)\nu}} \ge T
$$
for the prescribed value of $T$) for which we establish the existence of a unique spatially quasi-periodic solution to \eqref{sl}--\eqref{dt} up to the time horizon $T$.
\end{rema}

Furthermore, we may replace the exponential decay condition \eqref{exd} by the polynomial decay condition \eqref{pod} below and obtain the same conclusions for gBBM \eqref{sl} as in Theorem A. To express the time horizon as a function of the parameters of the decay parameters, it is convenient to introduce
$$
\mathfrak b(\mathtt s;\nu)\triangleq1+\sum_{j=1}^{\nu}\left(\begin{matrix} \nu \\j \end{matrix}\right)2^{j} j^{-\mathtt s}\left\{\zeta\left(\frac{\mathtt s}{j}\right)\right\}^{j},
$$
where $\zeta$ is the Riemann zeta function,
$$
\zeta(\mathtt s) = \sum_{n=1}^{\infty}\frac{1}{n^{\mathtt s}}.
$$

\begin{theob}
Suppose the Fourier coefficients $\hat u(n)$ of the initial data satisfy the following polynomial decay condition,
\begin{align}\label{pod}
|\hat u(n)|\leq\mathtt A^{\frac{1}{p-1}}(1+|n|)^{-\mathtt r},\quad\forall n\in\mathbb Z^\nu,
\end{align}
where $\mathtt A>0$, $\mathtt r$ is a sufficiently large positive constant, and $1\leq\nu<\frac{\mathtt r}{4}-2$.

Then the quasi-periodic Cauchy problem \eqref{sl}--\eqref{dt} has a unique spatially quasi-periodic solution in the classical sense on the time interval $[0,{\mathcal L}_p']$, where
\begin{align}\label{gaa}
{\mathcal L}_p' \triangleq \left(1-\frac{1}{p}\right)^{p-1} \mathtt A^{-1} \mathfrak b\left(\frac{\mathtt r}{2};\nu\right)^{-1}.
\end{align}
\end{theob}

\begin{rema}\label{rem.1.2}
(a) What was pointed out in Remark~\ref{rem.1.1} applies equally well here. Theorem~B is a local existence and uniqueness result with arbitrary time horizon. That is, given any $T > 0$, we provide an explicit class of quasi-periodic initial data with polynomial Fourier decay (namely those obeying \eqref{pod} with parameters $\mathtt A, \mathtt r, \nu$ subject to the condition
$$
\left(1-\frac{1}{p}\right)^{p-1} \mathtt A^{-1} \mathfrak b\left(\frac{\mathtt r}{2};\nu\right)^{-1} \ge T
$$
for the prescribed value of $T$) for which we establish the existence of a unique spatially quasi-periodic solution to \eqref{sl}--\eqref{dt} up to the time horizon $T$.

\medskip

(b) It follows from the exponential (resp. polynomial) decay that the solution we construct is in the classical sense. In addition, the exponential decay property implies that our result can be viewed as a Cauchy-Kovalevskaya theorem for the gBBM equation with quasi-periodic initial data, based on a basic fact: a quasi-periodic Fourier series with exponentially decaying Fourier coefficients is analytic.

\medskip

(c) The extension of the existence result for BBM with smooth and decaying initial data from \cite{BBM72} to the case of gBBM was discussed by Albert in \cite{A86}, see also \cite{A89}. Moreover, solutions for gBBM with $p = 5$ that are periodic in space and quasi-periodic in time were discussed by Shi and Yan in \cite{SY22}. Let us also mention that Wang discussed in \cite{W19} solutions to nonlinear PDEs that are periodic in space and quasi-periodic in time from a more general perspective.

\medskip

(d) The dependence on the spatial variable in our setting is neither decaying nor periodic. There are only a few existing results for initial data lacking these two properties. In addition to the works already mentioned, Oh discusses the nonlinear Schr\"odinger equation in one dimension with almost periodic initial data \cite{O15a, O15b} and Wang presents spatially quasi-periodic standing wave solutions to the nonlinear Schr\"odinger equation in arbitrary dimension \cite{W20CMP}.
We also refer the reader to \cite{DSS20, DSS21, KMV20} for a broader discussion and to \cite{EVY19, GH03, K18, LY20}, which are primarily based on inverse spectral theory and the preservation of reflectionlessness by equations in the KdV hierarchy (see also \cite{BE97, E93, E94, VY02} for related work).

\medskip

(e) The absence of decay and periodicity makes the problem at hand significantly more difficult. As in the works \cite{DG16JAMS} and \cite{DLX21ARXIV} we have to deal with the higher dimensional discrete convolution operation
\[
{\hat u}^{\ast p}(\text{\rm  fixed total distance})=\sum_{\substack{q_1,\cdots,q_{p}\in\mathbb Z^\nu\\q_1+\cdots+
q_p=~\text{\rm  fixed total distance}}}\prod_{j=1}^{p}{\hat u}(q_j)
\]
appearing in the Picard iteration, during which the number of terms will increase exponentially. More precisely, let $\mathbf N_k$ be the number of terms for the Picard sequence. It is easy to see that $\mathbf N_1=2$ and $\mathbf N_k=1+\mathbf N_{k-1}^p$ for all $k\geq2$. The key point to overcoming this difficulty is an explicit combinatorial analysis in order to obtain the exponential (resp. polynomial) decay of the Picard sequence; see \cite{DG16JAMS} and \cite{DLX21ARXIV} for an implementation of this strategy for KdV and gKdV, respectively.

\medskip

(f) The structure of the proofs of Theorem A and Theorem B is given by the following diagram:
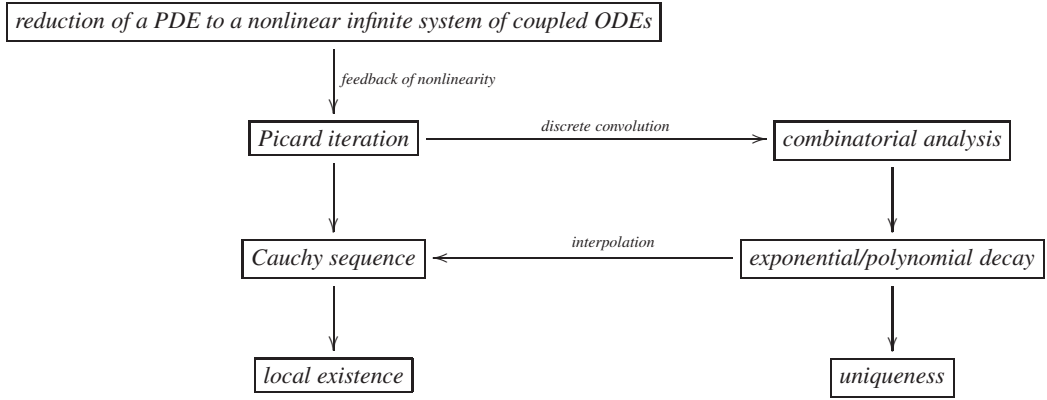
\[
\small
\xymatrix{
\boxed{\text{reduction of a PDE to a nonlinear infinite system of coupled ODEs}}\ar[d]^{\text{feedback of nonlinearity}}\\
\boxed{\text{Picard iteration}}\ar[d]\ar[r]^{\text{discrete convolution\hspace{2mm}}}&\boxed{\text{combinatorial analysis}}\ar[d]\\
\boxed{\text{Cauchy sequence}}\ar[d]&\boxed{\text{exponential/polynomial decay}}\ar[l]_{\text{interpolation}}\ar[d]\\
\boxed{\text{local existence}}&\boxed{\text{uniqueness}}
}
\]
\end{rema}

\section{The Special Case $p = 2$: BBM}\label{bm}

For the sake of convenience and readability, we first study the quasi-periodic Cauchy problem \eqref{sl}--\eqref{dt} for $p=2$. Whenever we refer to \eqref{sl} in this section we tacitly assume that $p=2$.

We will denote the Fourier coefficients at time $0$ and time $t$ by $c(n)\triangleq \hat  u(n)$ and $c(t,n)\triangleq\hat u(t,n)$, respectively.

\subsection{Reduction}

The first step in our proof is a reduction of the PDE in question to a nonlinear infinite system of coupled ODEs. For the latter we then consider a suitable Picard sequence.

Formally, by the Cauchy product for infinite series (i.e., the discrete convolution operation), we have
\begin{align}
(u^2)(t,x)=\sum_{n\in\mathbb Z^\nu}\sum_{\substack{n_1,n_2\in\mathbb Z^\nu\\n_1+n_2=n}}\prod_{j=1}^2c(t,n_j)e^{{\rm i}\langle n\rangle x}.
\end{align}
Assuming that $\partial$ and $\sum$ can be interchanged, we have
\begin{subequations}
\begin{align}
\label{a}u_t-u_{xxt}&=\sum_{n\in\mathbb Z^\nu}(1+\langle n\rangle^2)({\partial_tc})(t,n)e^{{\rm i}\langle n\rangle x},\\
\label{b}u_x&=\sum_{n\in\mathbb Z^\nu}{\rm i}\langle n\rangle c(t,n) e^{{\rm i}\langle n\rangle x},\\
\label{cc}uu_x&=\partial_x\left(\frac{u^2}{2}\right)=\sum_{n\in\mathbb Z^\nu}\frac{{\rm i}\langle n\rangle}{2}\sum_{\substack{n_1,n_2\in\mathbb Z^\nu\\n_1+n_2=n}}\prod_{j=1}^2c(t,n_j)e^{{\rm i}\langle n\rangle x}.
\end{align}
\end{subequations}
Substituting \eqref{a}-\eqref{cc} into \eqref{sl} yields
\[
\sum_{n\in\mathbb Z^\nu}\left\{(1+\langle n\rangle^2)(\partial_tc)(t,n)+{\rm i}\langle n\rangle c(t,n)+\frac{{\rm i}\langle n\rangle}{2}\sum_{\substack{n_1,n_2\in\mathbb Z^\nu\\n_1+n_2=n}}\prod_{j=1}^2c(t,n_j)\right\}e^{{\rm i}\langle n\rangle x}=0.
\]
By the orthogonality of $\{e^{{\rm i}\langle n\rangle x}: x\in\mathbb R\}$ relative to
\[<u,v>_{L^2(\mathbb R)}:=\lim_{L\rightarrow+\infty}\frac{1}{2L}\int_{-L}^{+L}u(x)\bar v(x){\rm d}x,\]
we see that \eqref{sl} is equivalent to the nonlinear infinite system of coupled ODEs
\begin{align}\label{ode}
\frac{{\rm d}}{{\rm d}t}c(t,n)-\lambda(n)c(t,n)=\frac{\lambda(n)}{2}\sum_{\substack{n_1,n_2\in\mathbb Z^\nu\\ n_1+n_2=n}}\prod_{j=1}^2c(t,n_j),
\end{align}
where
\begin{align}
\lambda(n)\triangleq\frac{-{\rm i}\langle n\rangle}{1+\langle n\rangle^2}\quad(\text{a purely imaginary number})
\end{align}
obeys the uniform bound $|\lambda(n)|\leq\frac{1}{2}$ for all $n\in\mathbb Z^\nu$. Here we use $``\frac{{\rm d}}{{\rm d}t}"$ rather than $``\partial_t"$ to emphasize that \eqref{ode} is an ODE for any given $n\in\mathbb Z^\nu$.

Motivated by an idea from \cite{KPV91JAMS}, we observe that $c(t,n)$ is determined by the following integral equation,
\begin{align}\label{ie}
c(t,n)=e^{\lambda(n)t}c(n)+\frac{\lambda(n)}{2}\int_0^te^{\lambda(n)(t-\tau)}\sum_{\substack{n_1,n_2\in\mathbb Z^\nu\\n_1+n_2=n}}\prod_{j=1}^2c(\tau,n_j) \, {\rm d}\tau.
\end{align}

To determine $c(t,n)$, we construct a Picard sequence $\{c_k(t,n)\}_{k\geq0}$ to approximate it. We choose $e^{\lambda(n)t}c(n)$ as the initial guess $c_0(t,n)$ and obtain $\{c_k(t,n)\}_{k\geq1}$ via the following iteration,
\begin{align}\label{pi}
c_k(t,n):=
c_0(t,n)+\frac{\lambda(n)}{2}\int_0^te^{\lambda(n)(t-\tau)}\sum_{\substack{n_1,n_2\in\mathbb Z^\nu\\n_1+n_2=n}}\prod_{j=1}^2c_{k-1}(\tau,n_j) \, {\rm d}\tau,\quad\forall k\geq1.
\end{align}

\subsection{Combinatorial Tree for the Picard Sequence}

Our goal is to show that the Picard sequence converges. To this end, it is convenient to express it via a combinatorial tree. This is the aim of the present subsection.

Set
\begin{align*}
\spadesuit^{(k)}&:=
\begin{cases}
\{0,1\}, & k=1;\\
\{0\}\cup(\spadesuit^{(k-1)})^2,&k\geq2.
\end{cases}
\end{align*}

For $\gamma^{(k)}=0\in\spadesuit^{(k)},k\geq1$, $
\mathfrak N^{(k,0)}:=\mathbb Z^\nu$;  for $\gamma^{(1)}=1\in\spadesuit^{(1)}$, $\mathfrak N^{(1,1)}:=(\mathbb Z^\nu)^2$;
for $\gamma^{(k)}=(\gamma_1^{(k-1)},\gamma_2^{(k-1)})\in(\spadesuit^{(k-1)})^2, k\geq2$, $\mathfrak N^{(k,\gamma^{(k)})}:=\prod_{j=1}^2\mathfrak N^{(k-1,\gamma_j^{(k-1)})}$.

Define a function $\mu: (\mathbb Z^\nu)^\ast\rightarrow\mathbb Z^\nu$ by letting $\mu(\clubsuit)=\sum_{j=1}^\ast\clubsuit_j$, where $\clubsuit=(\clubsuit_j\in\mathbb Z^\nu)_{1\leq j\leq \ast}$ for all $j=1,\cdots,\ast\in\mathbb N_{+}$.

For $k\geq1, \gamma^{(k)}=0\in\spadesuit^{(k)}, n=n^{(k)}\in\mathfrak N^{(k,0)}$,
\begin{align*}
\mathfrak C^{(k,0)}(n^{(k)})&:=c(n),\\
\mathfrak I^{(k,0)}(t,n^{(k)})&:=e^{\lambda(n)t},\\
\mathfrak F^{(k,0)}(n^{(k)})&:=1;
\end{align*}
for $k=1,\gamma^{(1)}=1\in\spadesuit^{(1)}$, $(n_1,n_2)=n^{(1)}\in\mathfrak N^{(1,1)}$,
\begin{align*}
\mathfrak C^{(1,1)}(n^{(1)})&:=\prod_{j=1}^2c(n_j),\\
\mathfrak I^{(1,1)}(t,n^{(1)})&:=\int_0^te^{\lambda(\mu(n^{(1)}))(t-\tau)}\prod_{j=1}^2e^{\lambda(n_j)\tau}{\rm d}\tau,\\
\mathfrak F^{(1,1)}(n^{(1)})&:=\frac{\lambda(\mu(n^{(1)}))}{2};
\end{align*}
for $k\geq2,\gamma^{(k)}=(\gamma_1^{(k-1)},\gamma_2^{(k-1)})\in(\spadesuit^{(k-1)})^2$,
\begin{align*}
\mathfrak C^{(k,\gamma^{(k)})}(n^{(k)})&:=\prod_{j=1}^2\mathfrak C^{(k-1,\gamma_j^{(k-1)})}(n_j^{(k-1)}),\\
\mathfrak I^{(k,\gamma^{(k)})}(t,n^{(k)})&:=\int_0^te^{\lambda(n)(t-\tau)}\prod_{j=1}^2\mathfrak I^{(k-1,\gamma_j^{(k-1)})}(\tau,n_j^{(k-1)}){\rm d}\tau,\\
\mathfrak F^{(k,\gamma^{(k)})}&:=\frac{\lambda(\mu(n^{(k)}))}{2}\prod_{j=1}^2\mathfrak F^{(k-1,\gamma_j^{(k-1)})}.
\end{align*}

\begin{lemm}\label{cthm}
The Picard sequence $\{c_k(t,n)\}$ can be reformulated as the following combinatorial tree,
\begin{align}\label{ct}
c_k(t,n)=\sum_{\gamma^{(k)}\in\spadesuit^{(k)}}\sum_{\substack{n^{(k)}\in\mathfrak N^{(k,\gamma^{(k)})}\\\mu(n^{(k)})=n}}
\mathfrak C^{(k,\gamma^{(k)})}(n^{(k)})\mathfrak I^{(k,\gamma^{(k)})}(t,n^{(k)})\mathfrak F^{(k,\gamma^{(k)})}(n^{(k)}),\quad \forall k\geq1.
\end{align}
\end{lemm}

\begin{proof}
We first notice that
\begin{align*}
c_0(t,n)=\sum_{\gamma^{(k)}=0\in\spadesuit^{(k)}}\sum_{\substack{n^{(k)}\in\mathfrak N^{(k,\gamma^{(k)})}\\\mu(n^{(k)})=n}}\mathfrak C^{(k,\gamma^{(k)})}(n^{(k)})\mathfrak I^{(k,\gamma^{(k)})}(t,n^{(k)})\mathfrak F^{(k,\gamma^{(k)})}(n^{(k)}),\quad \forall k\geq1.
\end{align*}
For $k=1$, we have
\begin{align*}
c_1(t,n)-c_0(t,n) & = \frac{\lambda(n)}{2}\int_0^te^{\lambda(n)(t-\tau)}\sum_{\substack{n_1,n_2\in\mathbb Z^\nu\\n_1+n_2=n}}\prod_{j=1}^2c_0(\tau,n_j) \, {\rm d}\tau \\
& = \sum_{\substack{n_1,n_2\in\mathbb Z^\nu\\n_1+n_2=n}}\prod_{j=1}^2c(n_j)\cdot\frac{\lambda(n)}{2}\cdot\int_0^te^{\lambda(n)(t-\tau)}\prod_{j=1}^2e^{\lambda(n_j)\tau} \, {\rm d}\tau \\
& = \sum_{\gamma^{(1)}=1\in\spadesuit^{(1)}}\sum_{\substack{n^{(1)}\in\mathfrak N^{(1,\gamma^{(1)})}\\\mu(n^{(1)})=n}}\mathfrak C^{(1,\gamma^{(1)})}(n^{(1)})\mathfrak I^{(1,\gamma^{(1)})}(t,n^{(1)})\mathfrak F^{(1,\gamma^{(1)})}(n^{(1)}).
\end{align*}
Hence we have
\begin{align*}
c_1(t,n) & = \left(\sum_{\gamma^{(1)}=0\in\spadesuit^{(1)}}+\sum_{\gamma^{(1)}=1\in\spadesuit^{(1)}}\right)\sum_{\substack{n^{(1)}\in\mathfrak N^{(1,\gamma^{(1)})}\\\mu(n^{(1)})=n}}\mathfrak C^{(1,\gamma^{(1)})}(n^{(1)})\mathfrak I^{(1,\gamma^{(1)})}(t,n^{(1)})\mathfrak F^{(1,\gamma^{(1)})}(n^{(1)}) \\
& = \sum_{\gamma^{(1)}\in\spadesuit^{(1)}}\sum_{\substack{n^{(1)}\in\mathfrak N^{(1,\gamma^{(1)})}\\\mu(n^{(1)})=n}}\mathfrak C^{(1,\gamma^{(1)})}(n^{(1)})\mathfrak I^{(1,\gamma^{(1)})}(t,n^{(1)})\mathfrak F^{(1,\gamma^{(1)})}(n^{(1)}).
\end{align*}
This shows that \eqref{ct} holds for $k=1$.

Let $k\geq2$ and assume that \eqref{ct} is true for $1,\cdots,k-1$. For $k$, we have
\begin{align*}
c_k(t,n)-c_0(t,n) & = \frac{\lambda(n)}{2}\int_0^te^{\lambda(n)(t-\tau)}\sum_{\substack{n_1,n_2\in\mathbb Z^\nu\\n_1+n_2=n}}\prod_{j=1}^2c_{k-1}(\tau,n_j) \, {\rm d}\tau \\
& = \frac{\lambda(n)}{2}\int_0^te^{\lambda(n)(t-\tau)}\sum_{\substack{n_1,n_2\in\mathbb Z^\nu\\n_1+n_2=n}}\prod_{j=1}^2\sum_{\gamma_j^{(k-1)}\in\spadesuit^{(k-1)}}\sum_{\substack{n_j^{(k-1)}\in\mathfrak N^{(k-1,\gamma_j^{(k-1)})}\\\mu(n_j^{(k-1)})=n_j}}\\
& \qquad \mathfrak C^{(k-1,\gamma_j^{(k-1)})}(n_j^{(k-1)})\mathfrak I^{(k,\gamma^{(k)})}(\tau,n_j^{(k-1)})\mathfrak F^{(k-1,\gamma_j^{(k-1)})}(n_j^{(k-1)}) \, {\rm d}\tau \\
& = \sum_{\substack{\gamma_j^{(k-1)} \in \spadesuit^{(k-1)} \\ j=1,2}}\sum_{\substack{n_1,n_2\in\mathbb Z^\nu\\n_1+n_2=n}}\sum_{\substack{n_j^{(k-1)}\in\mathfrak N^{(k-1,\gamma_j^{(k-1)})}\\\mu(n_j^{(k-1)})=n_j\\j=1,2}}\prod_{j=1}^2\mathfrak C^{(k-1,\gamma_j^{(k-1)})}(n_j^{(k-1)})\cdot \\
& \qquad \frac{\lambda(n)}{2}\prod_{j=1}^2\mathfrak F^{(k-1,\gamma_j^{(k-1)})}(n_j^{(k-1)})\cdot\int_0^te^{\lambda(n)(t-\tau)}\prod_{j=1}^2\mathfrak I^{(k-1,\gamma_j^{(k-1)})}(\tau,n_j^{(k-1)}) \, {\rm d}\tau \\
& = \sum_{\gamma^{(k)}\in(\spadesuit^{(k-1)})^2}\sum_{\substack{n^{(k)}\in\mathfrak N^{(k,\gamma^{(k)})}\\\mu(n^{(k)})=n}}\mathfrak C^{(k,\gamma^{(k)})}(n^{(k)})\mathfrak I^{(k,\gamma^{(k)})}(t,n^{(k)})\mathfrak F^{(k,\gamma^{(k)})}(n^{(k)}).
\end{align*}
Thus we have
\begin{align*}
c_k(t,n) & = \left( \sum_{\gamma^{(k)}=0\in\spadesuit^{(k)}}  +\sum_{\gamma^{(k)}\in(\spadesuit^{(k-1)})^2}\right)\sum_{\substack{n^{(k)}\in\mathfrak N^{(k,\gamma^{(k)})} \\\mu(n^{(k)})=n}}\mathfrak C^{(k,\gamma^{(k)})}(n^{(k)})\mathfrak I^{(k,\gamma^{(k)})}(t,n^{(k)})\mathfrak F^{(k,\gamma^{(k)})}(n^{(k)}) \\
& = \sum_{\gamma^{(k)}\in\spadesuit^{(k)}}\sum_{\substack{n^{(k)}\in\mathfrak N^{(k,\gamma^{(k)})}\\\mu(n^{(k)})=n}}
\mathfrak C^{(k,\gamma^{(k)})}(n^{(k)})\mathfrak I^{(k,\gamma^{(k)})}(t,n^{(k)})\mathfrak F^{(k,\gamma^{(k)})}(n^{(k)}).
\end{align*}
This shows that \eqref{ct} holds for $k$. By induction, it follows that \eqref{ct} is true for all $k\geq1$. This completes the proof of Lemma \ref{cthm}.
\end{proof}

\subsection{Uniform Exponential Decay of the Picard Sequence}

In this subsection we show with the help of the combinatorial tree established in the previous subsection that the Picard sequence for the Fourier coefficients obeys a uniform exponential decay estimate.

Indeed, we have the following result:

\begin{lemm}\label{expthm}
Assume that the initial Fourier coefficients $c$ satisfy the exponential decay property \eqref{exd}. With the constants $\mathcal A$ and $\rho$ from \eqref{exd} and the dimension $\nu$, set
\begin{equation}\label{e.newconstant}
\mathcal B \triangleq 2 \mathcal A (6 \rho^{-1})^\nu
\end{equation}
and
\begin{equation}\label{e.locextime}
{\mathcal L}_2 \triangleq \frac{\rho^\nu}{2\mathcal A6^\nu}.
\end{equation}
Then, we have
\begin{equation}\label{cktnunifupperbound}
\sup_{\substack{t\in[0,{\mathcal L}_2] \\ k \ge 0}} |c_k(t,n)|\leq\mathcal Be^{-\frac{\rho}{2}|n|}
\end{equation}
for every $n \in \mathbb Z^\nu$.
\end{lemm}

To prove Lemma \ref{expthm}, we need the following lemmas.
\begin{lemm}\label{ind}
For all $k\geq1$ we have
\begin{align}
\label{c}|\mathfrak C^{(k,\gamma^{(k)})}(n^{(k)})|&\leq\mathcal A^{\sigma(\gamma^{(k)})}e^{-\rho|n^{(k)}|},\\
\label{i}|\mathfrak I^{(k,\gamma^{(k)})}(t,n^{(k)})|&\leq\frac{t^{\ell(\gamma^{(k)})}}{\mathfrak D(\gamma^{(k)})},\\
\label{f}|\mathfrak F^{(k,\gamma^{(k)})}(n^{(k)})|&\leq\frac{1}{2^{\ell(\gamma^{(k)})}}\leq1,
\end{align}
where $\sigma(0)=1,\ell(0)=0,\mathfrak D(0)=1$ ; $\sigma(1)=2,\ell(1)=1,\mathfrak D(1)=1$; for $k\geq2,\gamma^{(k)}=(\gamma_1^{(k-1)},\gamma_2^{(k-1)})\in(\spadesuit^{(k-1)})^2$,
\begin{align*}
\sigma(\gamma^{(k)})&=\sum_{j=1}^2\sigma(\gamma_j^{(k-1)}),\\
\ell(\gamma_j^{(k)})&=1+\sum_{j=1}^2\ell(\gamma_j^{(k-1)}),\\
\mathfrak D(\gamma^{(k)})&=\ell(\gamma^{(k)})\prod_{j=1}^2\mathfrak D^{(k-1,\gamma_j^{(k-1)})}(n_j^{(k-1)}),
\end{align*}
and $|n^{(k)}|=\sum_{j=1}^\ast|n_j^{(k-1)}|$ if $n^{(k)}=(n_j^{(k-1)})_{1\leq j\leq\ast}$.
\end{lemm}
\begin{proof}
For $k\geq1, 0=\gamma^{(k)}\in\spadesuit^{(k)}, n=n^{(k)}\in\mathfrak N^{(k,0)}$,
\begin{align*}
|\mathfrak C^{(k,0)}(n^{(k)})|&=|c(n)|\leq\mathcal Ae^{-\rho|n|}=\mathcal A^{\sigma(0)}e^{-\rho|n^{(k)}|};\\
|\mathfrak I^{(k,0)}(t,n^{(k)})|&=|e^{\lambda(n)t}|\leq1=\frac{t^{\ell(0)}}{\mathfrak D(0)};\\
|\mathfrak F^{(k,0)}(n^{(k)})|&=1=\frac{1}{2^{\ell(0)}}\leq1.
\end{align*}
For $k=1, 1=\gamma^{(1)}\in\spadesuit^{(1)}, (n_1,n_2)=n^{(1)}\in\mathfrak N^{(1,1)}$,
\begin{align*}
|\mathfrak C^{(1,1)}(n_1,n_2)|&=\prod_{j=1}^2|c(n_j)|\leq\prod_{j=1}^2\mathcal Ae^{-\rho|n_j|}=\mathcal A^2e^{-\rho(|n_1|+|n_2|)}=\mathcal A^{\sigma(1)}e^{-\rho|n^{(1)}|};\\
|\mathfrak I^{(1,1)}(t,n^{(1)})|&\leq\int_0^t|e^{\lambda(\mu(n^{(1)}))(t-\tau)}|\prod_{j=1}^2|e^{\lambda(n_j)\tau}|{\rm d}\tau=t=\frac{t^{\ell(1)}}{\mathfrak D(1)};\\
|\mathfrak F^{(1,1)}(n^{(1)})|&=\frac{|\lambda(\mu(n^{(1)}))|}{2}\leq|\lambda(\mu(n^{(1)}))|\leq\frac{1}{2}=\frac{1}{2^{\ell(1)}}<1.
\end{align*}
Hence \eqref{c}--\eqref{f} hold for $k=1$.

Let $k\geq2$ and assume that they are true for $1,\cdots,k-1$.  For $k, (\gamma_1^{(k-1)},\gamma_2^{(k-1)})=\gamma^{(k)}\in(\spadesuit^{(k-1)})^2$ and $ (n_1^{(k-1)},n_2^{(k-1)})=n^{(k)}\in\prod_{j=1}^2\mathfrak N^{(k-1,\gamma_j^{(k-1)})}$, one can derive that
\begin{align*}
|\mathfrak C^{(k,\gamma^{(k)})}(n^{(k)})|&=\prod_{j=1}^2|\mathfrak C^{(k-1,\gamma_j^{(k-1)})}(n_j^{(k-1)})|\\
&\leq\prod_{j=1}^2\mathcal A^{\sigma(\gamma_j^{(k-1)})}e^{-\rho|n_j^{(k-1)}|}\\
&=\mathcal A^{\sum_{j=1}^2\sigma(\gamma_j^{(k-1)})}e^{-\rho\sum_{j=1}^2|n_j^{(k-1)}|}\\
&=\mathcal A^{\sigma(\gamma^{(k)})}e^{-\rho|n^{(k)}|};\\
|\mathfrak I^{(k,\gamma^{(k)})}(t,n^{(k)})|&\leq\int_0^t|e^{\lambda(\mu(n^{(k)}))(t-\tau)}|\prod_{j=1}^2|\mathfrak I^{(k-1,\gamma_j^{(k-1)})}(\tau,n_j^{(k-1)})|{\rm d}\tau\\
&\leq\int_0^t\prod_{j=1}^2\frac{\tau^{\ell(\gamma_j^{(k-1)})}}{\mathfrak D(\gamma_j^{(k-1)})}{\rm d}\tau\\
&=\frac{t^{1+\sum_{j=1}^2}\ell(\gamma_j^{(k-1)})}{(1+\sum_{j=1}^2\ell(\gamma_j^{(k-1)}))\prod_{j=1}^2\mathfrak D(\gamma_j^{(k-1)})}\\
&=\frac{t^{\ell(\gamma^{(k)})}}{\mathfrak D(\gamma^{(k)})};\\
|\mathfrak F^{(k,\gamma^{(k)})}(n^{(k)})|&\leq\frac{|\lambda(\mu(n^{(k)}))|}{2}\prod_{j=1}^2|\mathfrak F^{(k-1,\gamma_j^{(k-1)})}(n_j^{(k-1)})|\\
&\leq\frac{1}{2}\prod_{j=1}^2\frac{1}{2^{\ell(\gamma_j^{(k-1)})}}\\
&=\frac{1}{2^{1+\sum_{j=1}^2\ell(\gamma_j^{(k-1)})}}\\
&=\frac{1}{2^{\ell(\gamma^{(k)})}}\\
&<1.
\end{align*}
These imply that \eqref{c}--\eqref{f} are true for $k$. By induction, they hold for all $k\geq1$.  This completes the proof of Lemma \ref{ind}.
\end{proof}

\begin{lemm}\label{l.Lemma2.4}
\begin{enumerate}
  \item For all $k\geq1$,
  \begin{align}\label{1}
  \sigma(\gamma^{(k)})=\ell(\gamma^{(k)})+1.
  \end{align}
  \item If $0<\rho\leq 1$, then
  \begin{align}\label{2}
  \sum_{n\in\mathbb Z}e^{-\rho|n|}\leq 3\rho^{-1}.
  \end{align}
   \item Let $\dim_{\mathbb Z^\nu}\mathfrak N^{(k,\gamma^{(k)})}$ be the number of components per $\mathbb Z^\nu$. Then
   \begin{align}\label{3}
   \dim_{\mathbb Z^\nu}\mathfrak N^{(k,\gamma^{(k)})}=\sigma(\gamma^{(k)}).
   \end{align}
   \item If $0<\flat\leq\frac{1}{4}$, then
   \begin{align}\label{4}
  \diamondsuit_k\triangleq\sum_{\gamma^{(k)}\in\spadesuit^{(k)}}\frac{\flat^{\ell(\gamma^{(k)})}}{\mathfrak D(\gamma^{(k)})}\leq2.
   \end{align}
\end{enumerate}
\end{lemm}

\begin{proof}
\begin{enumerate}

\item Since $\sigma(0)=1,\ell(0)=0; \sigma(1)=2,\ell(1)=1$, then $\sigma(0)=\ell(0)+1$ and $\sigma(1)=\ell(1)+1$. Hence \eqref{1} holds for $k=1$. Let $k\geq2$. Assume that is is true for $1,\cdots,k-1$. For $k$ and $(\gamma_1^{(k-1)},\gamma_2^{(k-1)})=\gamma^{(k)}\in(\spadesuit^{(k-1)})^2$, one has
      \begin{align*}
      \sigma(\gamma^{(k)})
      =\sum_{j=1}^2\sigma(\gamma_j^{(k-1)})
      =\sum_{j=1}^2(\ell_j^{(k-1)}+1)
      =1+\ell(\gamma^{(k)}).
      \end{align*}
Hence \eqref{1} holds for all $k\geq1$ by induction.

\item Let $z(y) := (3-y)e^y-(3+y), 0<y\leq1$. After a simple calculation, we find
\begin{align*}
   z^\prime(y)=(2-y)e^y-1,\quad z^{\prime\prime}(y)=(1-y)e^y.
\end{align*}
  Since $0<y\leq1$, then $z^{\prime\prime}(y)\geq0$. Hence $z^\prime$ is monotonically increasing and $z^\prime(y)\geq z^\prime(0)=1>0$. Similarly one can see that $z(y)\geq z(0)=0$, that is $\frac{e^y+1}{e^{y}-1}\leq 3y^{-1}$ provided that $0<y\leq1$. By the symmetry of $\mathbb Z$ and $0<\rho\leq1$,
  \begin{align}
  \sum_{n\in\mathbb Z}e^{-\rho|n|}=2\sum_{n=0}^{\infty}e^{-\rho n}-1\stackrel{(\rho>0)}{=}\frac{2}{1-e^{-\rho}}-1=\frac{e^{\rho}+1}{e^{\rho}-1}\stackrel{(0<\rho\leq1)}{\leq}3\rho^{-1}.
  \end{align}

\item For $k\geq1, \mathfrak N^{(k,0)}=\mathbb Z^\nu$, we know that $\dim_{\mathbb Z^\nu}\mathfrak N^{(k,0)}=1=\sigma(0)$. Also $\mathfrak N^{(1,1)}=(\mathbb Z^\nu)^2$, hence $\dim_{\mathbb Z^\nu}\mathfrak N^{(1,1)}=2=\sigma(1)$. This implies that \eqref{3} holds for $k=1$. Let $k\geq2$. Assume that it is true for $1,\cdots,k-1$. For $k$, $\gamma^{(k)}=(\gamma_1^{(k-1)},\gamma_2^{(k-1)})\in(\spadesuit^{(k-1)})^2$, and $\mathfrak N^{(k,\gamma^{(k)})}=\prod_{j=1}^2\mathfrak N^{(k-1,\gamma_j^{(k-1)})}$, we have
      \begin{align*}
      \dim_{\mathbb Z^\nu}\mathfrak N^{(k,\gamma^{(k)})}=\sum_{j=1}^2\dim_{\mathbb Z^\nu}\mathfrak N^{(k-1,\gamma_j^{(k-1)})}=\sum_{j=1}^2\sigma(\gamma_j^{(k-1)})=\sigma(\gamma^{(k)}).
      \end{align*}
      By induction, \eqref{3} holds for all $k\geq1$.

\item For $k=1$, by the definition of $\spadesuit^{(1)}, \ell$ and $\mathfrak D$, we have
      \begin{align*}
      \diamondsuit_1=\frac{\flat^{\ell(0)}}{\mathfrak D(0)}+\frac{\flat^{\ell(1)}}{\mathfrak D(1)}=1+\flat\leq\frac{5}{4}\leq2.
      \end{align*}
      Let $k\geq2$. Suppose that \eqref{4} holds for $1,\cdots,k-1$. For $k$, we first have
\begin{align*}
\sum_{\gamma^{(k)}=(\gamma_1^{(k-1)},\gamma_2^{(k-1)})\in(\spadesuit^{(k-1)})^2}\frac{\flat^{\ell(\gamma^{(k)})}}{\mathfrak D(\gamma^{(k)})} & = \sum_{\substack{\gamma_j^{(k-1)}\in\spadesuit^{(k-1)}\\j=1,2}}
      \frac{\flat^{1+\sum_{j=1}^2\ell(\gamma_j^{(k-1)})}}{\left(1+\sum_{j=1}^2\gamma_j^{(k-1)}\right)\prod_{j=1}^2\mathfrak D(\gamma_j^{(k-1)})} \\
& \leq \flat\sum_{\substack{\gamma_j^{(k-1)}\in\spadesuit^{(k-1)}\\j=1,2}}\prod_{j=1}^2\frac{\flat^{\ell(\gamma_j^{(k-1)})}}{\mathfrak D^{(k-1,\gamma_j^{(k-1)})}(n_j^{(k-1)})} \\
& = \flat\prod_{j=1}^2\sum_{\gamma_j^{(k-1)}\in\spadesuit^{(k-1)}}\frac{\flat^{\ell(\gamma_j^{(k-1)})}}{\mathfrak D^{(k-1,\gamma_j^{(k-1)})}(n_j^{(k-1)})} \\
& \leq 2^2\flat.
\end{align*}
Hence we have
\begin{align*}
\diamondsuit_k & = \frac{\flat^{\ell(0)}}{\mathfrak D(0)}+\sum_{\gamma^{(k)}=(\gamma_1^{(k-1)},\gamma_2^{(k-1)})\in(\spadesuit^{(k-1)})^2}\frac{\flat^{\ell(\gamma^{(k)})}}{\mathfrak D(\gamma^{(k)})} \\
& \leq 1+2^2\flat \\
& \leq 2.
\end{align*}
This shows that \eqref{4} is true for $k$. By induction, \eqref{4} holds for all $k\geq1$.

\end{enumerate}
\end{proof}

\begin{rema}\label{r.optimalchoices}
The value $\frac{1}{4}$ as the upper limit for the range of $\flat$ and the value $2$ as the upper bound for $\diamondsuit_k$ in part (4) of Lemma~\ref{l.Lemma2.4} cannot be improved. In fact, assuming that $\diamondsuit_k \leq M$ for some $M>1$ and $\flat$ from a suitable interval $(0,r]$, the induction argument yields $1+\flat M^2\leq M$, and hence $0<\flat\leq1/M-(1/M)^2$. Consider the following auxiliary  function defined by letting $f(x)=x-x^2$, where $0<x<1$. Clearly, $f$ takes its maximum at $x = 1/2$. This implies that $M=2$ and $0<\flat\leq1/4$ are optimal.
\end{rema}

\begin{proof}[Proof of Lemma \ref{expthm}]
If $0\leq t\leq{\mathcal L}_2$, then
\begin{eqnarray*}
|c_k(t,n)| &\stackrel{\eqref{ct}}{\leq}&\sum_{\gamma^{(k)}\in\spadesuit^{(k)}}\sum_{\substack{n^{(k)}\in\mathfrak N^{(k,\gamma^{(k)})}\\\mu(n^{(k)})}=n}|\mathfrak C^{(k,\gamma^{(k)})}(n^{(k)})||\mathfrak I^{(k,\gamma^{(k)})}(t,n^{(k)})||\mathfrak F^{(k,\gamma^{(k)})}(n^{(k)})|\\
&\stackrel{\eqref{c}-\eqref{f}}{\leq}&\mathcal A\sum_{\gamma^{(k)}\in\spadesuit^{(k)}}\frac{(2^{-1}\mathcal At)^{\ell(\gamma^{(k)})}}{\mathfrak D(\gamma^{(k)})}\sum_{\substack{n^{(k)}\in\mathfrak N^{(k,\gamma^{(k)})}\\\mu(n^{(k)})=n}}e^{-\rho|n^{(k)}|}\\
&\stackrel{}{\leq}&\mathcal A\sum_{\gamma^{(k)}\in\spadesuit^{(k)}}\frac{(2^{-1}\mathcal At)^{\ell(\gamma^{(k)})}}{\mathfrak D(\gamma^{(k)})}\sum_{\substack{n^{(k)}\in\mathfrak N^{(k,\gamma^{(k)})}}}e^{-\frac{\rho}{2}|n^{(k)}|}\cdot e^{-\frac{\rho}{2}|n|}\\
&\stackrel{\eqref{3}}{\leq}&\mathcal A\sum_{\gamma^{(k)}\in\spadesuit^{(k)}}\frac{(2^{-1}\mathcal At)^{\ell(\gamma^{(k)})}}{\mathfrak D(\gamma^{(k)})}\sum_{n^{(k)}\in(\mathbb Z^\nu)^{\sigma(\gamma^{(k)})}}e^{-\frac{\rho}{2}|n^{(k)}|}\cdot e^{-\frac{\rho}{2}|n|}\\
&\stackrel{}{=}&\mathcal A\sum_{\gamma^{(k)}\in\spadesuit^{(k)}}\frac{(2^{-1}\mathcal At)^{\ell(\gamma^{(k)})}}{\mathfrak D(\gamma^{(k)})}\prod_{j=1}^{\sigma(\gamma^{(k)})}\prod_{j^\prime=1}^\nu \sum_{n_{jj^\prime}\in\mathbb Z}e^{-\frac{\rho}{2}|n_{jj^\prime}|}\cdot e^{-\frac{\rho}{2}|n|}\\
&\stackrel{\eqref{2}}{\leq}&\mathcal A\sum_{\gamma^{(k)}\in\spadesuit^{(k)}}\frac{(2^{-1}\mathcal At)^{\ell(\gamma^{(k)})}}{\mathfrak D(\gamma^{(k)})}(6\rho^{-1})^{\sigma(\gamma^{(k)})\nu}\cdot e^{-\frac{\rho}{2}|n|}\\
&\stackrel{\eqref{1}}{=}&\mathcal A(6\rho^{-1})^\nu\sum_{\gamma^{(k)}\in\spadesuit^{(k)}}\frac{\left(2^{-1}\mathcal A(6\rho^{-1})^{\nu}t\right)^{\ell(\gamma^{(k)})}}{\mathfrak D(\gamma^{(k)})}\cdot e^{-\frac{\rho}{2}|n|}\\
&\stackrel{\eqref{4}}{\leq}&\mathcal Be^{-\frac{\rho}{2}|n|}, \quad\text{where $\mathcal B$ is given by \eqref{e.newconstant}, i.e., } \mathcal B=2\mathcal A(6\rho^{-1})^\nu.
\end{eqnarray*}
This completes the proof of Lemma \ref{expthm}.
\end{proof}

\subsection{Convergence of the Picard Sequence}

We are now in a position to show that the Picard sequence for the Fourier coefficients converges. The following lemma establishes a bound for the magnitude of $c_k(t,n)-c_{k-1}(t,n)$, from which the desired conclusion follows.

\begin{lemm}\label{cslem}
For $k\geq1$, $0 \leq t \leq{\mathcal L}_2$ with ${\mathcal L}_2$ from \eqref{e.locextime}, and $n \in \mathbb Z^\nu$, we have
\begin{align}
\label{cs} |c_k(t,n)-c_{k-1}(t,n)| & \leq \frac{2^{k-1}\mathcal B^{k+1}t^k}{4^k\cdot k!}\sum_{\substack{n_1,\cdots,n_{k+1}\in\mathbb Z^\nu\\n_1+\cdots+n_{k+1}=n}}\prod_{j=1}^{k+1}e^{-\frac{\rho}{2}|n_j|} \\
\label{cs2} & \leq \frac{\mathcal B(12\rho^{-1})^\nu}{2}\cdot\frac{\left(2^{-1}\mathcal B(12\rho^{-1})^\nu t\right)^k}{k!}\cdot e^{-\frac{\rho}{4}|n|}.
\end{align}
Hence, $\{c_k(t,n)\}$ is a Cauchy sequence.
\end{lemm}

\begin{proof}
For $k=1$ we have
\begin{align*}
|c_1(t,n)-c_0(t,n)| & \leq \frac{|\lambda(n)|}{2}\int_0^t|e^{\lambda(n)(t-\tau)}|\sum_{\substack{n_1,n_2\in\mathbb Z^\nu\\n_1+n_2=n}}\prod_{j=1}^2|c_0(\tau,n_j)| \, {\rm d}\tau \\
& \leq \frac{\mathcal B^2t}{4}\sum_{\substack{n_1,n_2\in\mathbb Z^\nu\\n_1+n_2=n}}\prod_{j=1}^2e^{-\frac{\rho}{2}|n_j|}.
\end{align*}
This shows that \eqref{cs} holds for $k=1$.

Let $k\geq2$ and assume that \eqref{cs} is true for $1,\cdots,k-1$. For $k$, we first have
\begin{align*}
|c_{k}(t,n)-c_{k-1}(t,n)| & \leq \frac{|\lambda(n)|}{2}\int_0^t|e^{\lambda(n)(t-\tau)}|\sum_{\substack{n_1,n_2\in\mathbb Z^\nu\\n_1+n_2=n}} \Big| \prod_{j=1}^2c_{k-1}(\tau,n_j)-\prod_{j=1}^2c_{k-2}(\tau,n_j) \Big| \, {\rm d}\tau \\
& \leq \frac{1}{4}\int_0^t\sum_{\substack{n_1,n_2\in\mathbb Z^\nu\\n_1+n_2=n}}|c_{k-1}(\tau,n_1)||c_{k-1}(\tau,n_2)-c_{k-2}(\tau,n_2)| \, {\rm d}\tau\triangleq(\uppercase\expandafter{\romannumeral1}) \\
& \quad + \frac{1}{4}\int_0^t\sum_{\substack{n_1,n_2\in\mathbb Z^\nu\\n_1+n_2=n}}|c_{k-1}(\tau,n_1)-c_{k-2}(\tau,n_1)||c_{k-2}(\tau,n_2)| \, {\rm d}\tau\triangleq(\uppercase\expandafter{\romannumeral2}).
\end{align*}
For the first component, it follows from the induction hypothesis and Lemma~\ref{expthm} that
\begin{align*}
(\uppercase\expandafter{\romannumeral1})
& \leq \frac{1}{4}\int_0^t\sum_{\substack{n_1,n_2\in\mathbb Z^\nu\\n_1+n_2=n}}\mathcal Be^{-\frac{\rho}{2}|n_1|}\cdot\frac{2^{k-2}\mathcal B^{k}\tau^{k-1}}{4^{k-1}\cdot(k-1)!}\sum_{\substack{m_1,\cdots,m_k\in\mathbb Z^\nu\\m_1+\cdots+m_k=n_2}}\prod_{j=1}^{k}e^{-\frac{\rho}{2}|m_j|} \, {\rm d}\tau \\
& = \frac{2^{k-2}\mathcal B^{k+1}t^k}{4^k\cdot k!}\sum_{\substack{n_1,n_2\in\mathbb Z^\nu\\n_1+n_2=n}}\sum_{\substack{m_1,\cdots,m_k\in\mathbb Z^\nu\\m_1+\cdots+m_k=n_2}}e^{-\frac{\rho}{2}|n_1|}\prod_{j=1}^ke^{-\frac{\rho}{2}|m_j|} \\
& = \frac{2^{k-2}\mathcal B^{k+1}t^k}{4^k\cdot k!}\sum_{\substack{n_1,\cdots,n_{k+1}\in\mathbb Z^\nu\\n_1+\cdots+n_{k+1}=n}}\prod_{j=1}^{k+1}e^{-\frac{\rho}{2}|n_j|}.
\end{align*}
Analogously, for the second component we have
\[(\uppercase\expandafter{\romannumeral2})\leq\frac{2^{k-2}\mathcal B^{k+1}t^k}{4^k\cdot k!}\sum_{\substack{n_1,\cdots,n_{k+1}\in\mathbb Z^\nu\\n_1+\cdots+n_{k+1}=n}}\prod_{j=1}^{k+1}e^{-\frac{\rho}{2}|n_j|}.\]
Thus we find that
\begin{align*}
|c_k(t,n)-c_{k-1}(t,n)| & \leq (\uppercase\expandafter{\romannumeral1})+(\uppercase\expandafter{\romannumeral2}) \\
& \leq 2\times\frac{2^{k-2}\mathcal B^{k+1}t^k}{4^k\cdot k!}\sum_{\substack{n_1,\cdots,n_{k+1}\in\mathbb Z^\nu\\n_1+\cdots+n_{k+1}=n}}\prod_{j=1}^{k+1}e^{-\frac{\rho}{2}|n_j|} \\
& \leq \frac{2^{k-1}\mathcal B^{k+1}t^k}{4^k\cdot k!}\sum_{\substack{n_1,\cdots,n_{k+1}\in\mathbb Z^\nu\\n_1+\cdots+n_{k+1}=n}}\prod_{j=1}^{k+1}e^{-\frac{\rho}{2}|n_j|}.
\end{align*}
By induction, we see that \eqref{cs} holds for all $k\geq1$. Furthermore,
\begin{eqnarray*}
|c_k(t,n)-c_{k-1}(t,n)| &\leq&\frac{2^{k-1}\mathcal B^{k+1}t^k}{4^k\cdot k!}\sum_{\substack{n_1,\cdots,n_{k+1}\in\mathbb Z^\nu\\n_1+\cdots+n_{k+1}=n}}\prod_{j=1}^{k+1}e^{-\frac{\rho}{2}|n_j|}\\
&\leq&\frac{2^{k-1}\mathcal B^{k+1}t^k}{4^k\cdot k!}\sum_{n_1,\cdots,n_{k+1}\in\mathbb Z^\nu}\prod_{j=1}^{k+1}e^{-\frac{\rho}{4}|n_j|}\cdot e^{-\frac{\rho}{4}|n|}\\
&=&\frac{2^{k-1}\mathcal B^{k+1}t^k}{4^k\cdot k!}\prod_{j=1}^{k+1}\prod_{j^\prime=1}^\nu\sum_{n_{jj^\prime}\in\mathbb Z}e^{-\frac{\rho}{4}|n_{jj^\prime}|}\cdot e^{-\frac{\rho}{4}|n|}\\
&\stackrel{\eqref{2}}{\leq}&\frac{2^{k-1}\mathcal B^{k+1}t^k}{4^k\cdot k!}(12\rho^{-1})^{(k+1)\nu}\cdot e^{-\frac{\rho}{4}|n|}\\
&=&\frac{\mathcal B(12\rho^{-1})^\nu}{2}\cdot\frac{\left(2^{-1}\mathcal B(12\rho^{-1})^\nu t\right)^k}{k!}\cdot e^{-\frac{\rho}{4}|n|}.
\end{eqnarray*}
Hence $\{c_k(t,n)\}$ is a Cauchy sequence on $[0,{\mathcal L}_2]\times\mathbb Z^\nu$, completing the proof of Lemma~\ref{cslem}.
\end{proof}

\subsection{Proof of Theorem A: Existence}\label{ea}

In this subsection we prove the existence part of Theorem~A.

\begin{proof}
By Lemma \ref{cslem} we know that $\{c_k(t,n)\}$ is a Cauchy sequence and there exists a limit function, denoted by $c^\dag(t,n)$, where $0\leq t\leq{\mathcal L}_2$ and $n\in\mathbb Z^\nu$. By the triangle inequality we have
\begin{align*}
|c^\dag(t,n)|\leq|c^\dag(t,n)-c_k(t,n)|+|c_k(t,n)|, \quad\forall k\geq1.
\end{align*}
Since $k \ge 1$ is arbitrary in this estimate and the $c_k(t,n)$ obey the uniform upper bound \eqref{cktnunifupperbound}, it follows that the same upper bound holds for $c^\dag(t,n)$, that is,
\begin{align}\label{gf}
|c^\dag(t,n)|\leq\mathcal B e^{-\frac{\rho}{2}|n|}.
\end{align}

Naturally we regard the coefficients $c^\dag(t,n)$ as the Fourier coefficients of a candidate solution $u^\dag(t,x)$ of the Cauchy problem in question, and hence set
\begin{align*}
u^\dag(t,x)&:=\sum_{n\in\mathbb Z^\nu}c^\dag(t,n)e^{{\rm i}\langle n\rangle x};\\
(\partial_x^\#u^\dag)(t,x)&:=\sum_{n\in\mathbb Z^\nu}({\rm i\langle n\rangle})^{\#}c^\dag(t,n)e^{{\rm i}\langle n\rangle x}, \quad\#=1,2;\\
(u^\dag u^\dag_x)(t,x)&:=\sum_{n\in\mathbb Z^\nu}\frac{{\rm i}\langle n\rangle}{2}\sum_{\substack{n_1,n_2\in\mathbb Z^\nu\\n_1+n_2=n}}\prod_{j=1}^2c^{\dag}(t,n_j)e^{{\rm i}\langle n\rangle x};\\
(u^\dag_t)(t,x)&:=\sum_{n\in\mathbb Z^\nu}\left\{\lambda(n)c^\dag(t,n)+\frac{\lambda(n)}{2}\sum_{\substack{n_1,n_2\in\mathbb Z^\nu\\n_1+n_2=n}}\prod_{j=1}^2c^{\dag}(t,n_j)\right\}e^{{\rm i}\langle n\rangle x};\\
(u^\dag_{xxt})(t,x)&:=\sum_{n\in\mathbb Z^\nu}({\rm i}\langle n \rangle)^2\left\{\lambda(n)c^\dag(t,n)+\frac{\lambda(n)}{2}\sum_{\substack{n_1,n_2\in\mathbb Z^\nu\\n_1+n_2=n}}\prod_{j=1}^2c^{\dag}(t,n_j)\right\}e^{{\rm i}\langle n\rangle x}.
\end{align*}
We claim that $u^\dag$ is a classical spatially quasi-periodic solution to quasi-periodic Cauchy problem \eqref{sl}--\eqref{dt}, that is, $u^\dag_t, u^\dag_{xxt}, u^\dag_x, u^\dag u^\dag_x$ satisfy BBM \eqref{sl} in the classical sense and $u^\dag$ has initial data \eqref{dt}. On the one hand, by the exponential decay of $c^\dag(t,n)$, one can see that $u^\dag, u^\dag_t, u^\dag_{xxt}, u^\dag_x, u^\dag u^\dag_x$ are uniformly and absolutely convergent. It is sufficient to verify that
\begin{align*}
\sum_{n\in\mathbb Z^\nu}(1+|n|+|n|^2)\left\{c^\dag(t,n)+\sum_{\substack{n_1,n_2\in\mathbb Z^\nu\\n_1+n_2=n}}\prod_{j=1}^2|c^{\dag}(t,n_j)|\right\}<\infty.
\end{align*}
In fact, for $\#=0,1,2$, we have
\begin{eqnarray*}
\sum_{n\in\mathbb Z^\nu}|n|^{\#}|c^\dag(t,n)| &\stackrel{\eqref{gf}}{\lesssim}&\sum_{n\in\mathbb Z^\nu}|n|^{\#}e^{-\frac{\rho}{2}|n|}\\
&=&\sum_{n\in\mathbb Z^\nu}\underbrace{|n|^{\#}e^{-\frac{\rho}{4}|n|}}_{\text{bounded}}e^{-\frac{\rho}{4}|n|}\\
&\lesssim&\sum_{n\in\mathbb Z^\nu}e^{-\frac{\rho}{4}|n|}\\
&\stackrel{\eqref{2}}{\leq}&(12\rho^{-1})^{\nu}\\
&<&\infty
\end{eqnarray*}
and
\begin{eqnarray*}
\sum_{n\in\mathbb Z^\nu}|n|^{\#}\sum_{\substack{n_1,n_2\in\mathbb Z^\nu\\n_1+n_2=n}}\prod_{j=1}^2|c^{\dag}(t,n_j)| &\stackrel{\eqref{gf}}{\lesssim}&\sum_{n\in\mathbb Z^\nu}|n|^{\#}\sum_{\substack{n_1,n_2\in\mathbb Z^\nu\\n_1+n_2=n}}\prod_{j=1}^2e^{-\frac{\rho}{2}|n_j|}\\
&\leq&\sum_{n\in\mathbb Z^\nu}\underbrace{|n|^{\#}e^{-\frac{\rho}{8}|n|}}_{\text{bounded}}\underbrace{\sum_{n_1,n_2\in\mathbb Z^\nu}\prod_{j=1}^2e^{-\frac{\rho}{4}|n_j|}}_{\text{bounded by ~\eqref{2}}}e^{-\frac{\rho}{8}|n|}\\
&\lesssim&\sum_{n\in\mathbb Z^\nu}e^{-\frac{\rho}{8}|n|}\\
&\stackrel{\eqref{2}}{\leq}&(24\rho^{-1})^{\nu}\\
&<&\infty.
\end{eqnarray*}
On the another hand, since $c^\dag$ is the limit function of the Picard sequence $\{c_k\}$ defined by \eqref{pi}, it satisfies \eqref{ie}. Hence $c^\dag$ is a solution to \eqref{ode} and satisfies the initial condition $c^\dag(0,n)=c(n)$. This implies that $u^\dag$ is a classical spatially quasi-periodic solution to the Cauchy problem \eqref{sl}--\eqref{dt}. Hence the existence part of Theorem A is proved.
\end{proof}

\subsection{Proof of Theorem A: Uniqueness}\label{ua}

In this subsection we prove the uniqueness part of Theorem A.

\begin{proof}
Let
\[v(t,x)=\sum_{n\in\mathbb Z^\nu}\hat{v}(t,n)e^{{\rm i}\langle n\rangle x}\quad\text{and}\quad w(t,x)=\sum_{n\in\mathbb Z^\nu}\hat{w}(t,n)e^{{\rm i}\langle n\rangle x}\]
be two quasi-periodic solutions to \eqref{sl}--\eqref{dt}, where $\hat{v}$ and $\hat{w}$ satisfy the following conditions:
\begin{itemize}
  \item (same initial data)
  \[\hat{v}(0,n)=\hat{w}(0,n),\quad\forall n\in\mathbb Z^\nu;\]
  \item (integral equation)
  \begin{align*}
  \hat{v}(t,n)&=e^{\lambda(n)t}\hat{v}(0,n)+\frac{\lambda(n)}{2}\int_{0}^te^{\lambda(n)(t-\tau)}\sum_{\substack{n_1,n_2\in\mathbb Z^\nu\\n_1+n_2=n}}\prod_{j=1}^2\hat{v}(\tau,n_j){\rm d}\tau,\\
  \hat{w}(t,n)&=e^{\lambda(n)t}\hat{w}(0,n)+\frac{\lambda(n)}{2}\int_{0}^te^{\lambda(n)(t-\tau)}\sum_{\substack{n_1,n_2\in\mathbb Z^\nu\\n_1+n_2=n}}\prod_{j=1}^2\hat{w}(\tau,n_j){\rm d}\tau;
  \end{align*}
  \item (exponential decay)
  \[|\hat{v}(t,n)|\leq\mathcal Be^{-\frac{\rho}{2}|n|}\quad\text{and}\quad|\hat{w}(t,n)|\leq\mathcal B e^{-\frac{\rho}{2}|n|},\qquad0\leq t\leq\mathcal L_2, n\in\mathbb Z^\nu.\]
\end{itemize}
For all $k\geq1$, one can derive that
\begin{align}\label{u}
|\hat v(t,n)-\hat w(t,n)|\leq\frac{2^k\mathcal B^{k+1}t^k}{4^k\cdot k!}\sum_{\substack{n_1,\cdots,n_{k+1}\in\mathbb Z^\nu\\n_1+\cdots+n_{k+1}=n}}\prod_{j=1}^{k+1}e^{-\frac{\rho}{2}|n_j|}.
\end{align}
In fact,  we first have
\begin{align*}
|\hat v(t,n)-\hat w(t,n)| & \leq \frac{1}{4}\int_0^t\sum_{\substack{n_1,n_2\in\mathbb Z^\nu\\n_1+n_2=n}}|\prod_{j=1}^2\hat{v}(\tau,n_j)-\prod_{j=1}^2\hat w(\tau,n_j)| \, {\rm d}\tau \\
& \leq \frac{\mathcal B^2t}{2}\sum_{\substack{n_1,n_2\in\mathbb Z^\nu\\n_1+n_2=n}}\prod_{j=1}^2e^{-\frac{\rho}{2}|n_j|}.
\end{align*}
Hence \eqref{u} holds for $k=1$. Let $k\geq2$ and assume that it holds for $1,\cdots,k-1$. For $k$, we have
\begin{align*}
|\hat v(t,n)-\hat w(t,n)| & \leq \frac{1}{4}\int_0^t\sum_{\substack{n_1,n_2\in\mathbb Z^\nu\\n_1+n_2=n}} \Big| \prod_{j=1}^2\hat v(\tau,n_j)-\prod_{j=1}^2\hat w(\tau,n_j) \Big| \, {\rm d}\tau \\
& \leq \frac{1}{4}\int_0^t\sum_{\substack{n_1,n_2\in\mathbb Z^\nu\\n_1+n_2=n}}|\hat v(\tau,n_1)-\hat w(\tau,n_1)||\hat{v}(\tau,n_2)| \, {\rm d}\tau \triangleq(\uppercase\expandafter{\romannumeral1}^\prime) \\
& \quad + \frac{1}{4}\int_0^t\sum_{\substack{n_1,n_2\in\mathbb Z^\nu\\n_1+n_2=n}}|\hat{w}(\tau,n_1)||\hat v(\tau,n_2)-\hat w(\tau,n_2)| \, {\rm d}\tau\triangleq(\uppercase\expandafter{\romannumeral2}^\prime).
\end{align*}
For the first component we have
\begin{align*}
(\uppercase\expandafter{\romannumeral1}^\prime) & \leq \frac{1}{4}\int_0^t\sum_{\substack{n_1,n_2\in\mathbb Z^\nu\\n_1+n_2=n}}\frac{2^{k-1}\mathcal B^{k}\tau^{k-1}}{4^{k-1}\cdot (k-1)!}\sum_{\substack{m_1,\cdots,m_{k}\in\mathbb Z^\nu\\m_1+\cdots+m_{k}=n_1}}\prod_{j=1}^{k}e^{-\frac{\rho}{2}|m_j|}\cdot\mathcal Be^{-\frac{\rho}{2}|n_2|}{\rm d}\tau \\
& = \frac{2^{k-1}\mathcal B^{k+1}t^k}{4^k\cdot k!}\sum_{\substack{n_1,n_2\in\mathbb Z^\nu\\n_1+n_2=n}}\sum_{\substack{m_1,\cdots,m_{k}\in\mathbb Z^\nu\\m_1+\cdots+m_{k}=n_1}}e^{-\frac{\rho}{2}|n_2|}\prod_{j=1}^{k}e^{-\frac{\rho}{2}|m_j|} \\
& = \frac{2^{k-1}\mathcal B^{k+1}t^k}{4^k\cdot k!}\sum_{\substack{n_1,\cdots,n_{k+1}\in\mathbb Z^\nu\\n_1+\cdots+n_{k+1}=n}}\prod_{j=1}^{k+1}e^{-\frac{\rho}{2}|n_j|}.
\end{align*}
After a similar argument we find
\[(\uppercase\expandafter{\romannumeral2}^\prime)\leq\frac{2^{k-1}\mathcal B^{k+1}t^k}{4^k\cdot k!}\sum_{\substack{n_1,\cdots,n_{k+1}\in\mathbb Z^\nu\\n_1+\cdots+n_{k+1}=n}}\prod_{j=1}^{k+1}e^{-\frac{\rho}{2}|n_j|}.\]
Hence we have
\begin{align*}
|\hat v(t,n)-\hat w(t,n)| & \leq (\uppercase\expandafter{\romannumeral1}^\prime)+(\uppercase\expandafter{\romannumeral2}^\prime) \\
& \leq \frac{2^{k}\mathcal B^{k+1}t^k}{4^k\cdot k!}\sum_{\substack{n_1,\cdots,n_{k+1}\in\mathbb Z^\nu\\n_1+\cdots+n_{k+1}=n}}\prod_{j=1}^{k+1}e^{-\frac{\rho}{2}|n_j|}.
\end{align*}
By induction, \eqref{u} holds for all $k\geq1$. By \eqref{2} one can derive that
\begin{align*}
|\hat v(t,n)-\hat w(t,n)| & \leq \frac{2^{k}\mathcal B^{k+1}t^k}{4^k\cdot k!}\sum_{\substack{n_1,\cdots,n_{k+1}\in\mathbb Z^\nu\\n_1+\cdots+n_{k+1}=n}}\prod_{j=1}^{k+1}e^{-\frac{\rho}{2}|n_j|} \\
& = \frac{2^{k}\mathcal B^{k+1}t^k}{4^k\cdot k!}\sum_{n_1,\cdots,n_{k+1}\in\mathbb Z^\nu}\prod_{j=1}^{k+1}e^{-\frac{\rho}{4}|n_j|}\cdot e^{-\frac{\rho}{4}|n|} \\
& = \frac{2^{k}\mathcal B^{k+1}t^k}{4^k\cdot k!}\prod_{j=1}^{k+1}\prod_{j^\prime=1}^\nu\sum_{n_{jj^\prime}\in\mathbb Z^\nu}e^{-\frac{\rho}{4}|n_{jj^\prime}|}\cdot e^{-\frac{\rho}{4}|n|} \\
& \leq \frac{2^{k}\mathcal B^{k+1}t^k}{4^k\cdot k!}(12\rho^{-1})^{(k+1)\nu}\cdot e^{-\frac{\rho}{4}|n|} \\
& \leq \mathcal B(12\rho^{-1})^\nu\cdot\frac{(2^{-1}\mathcal B(12\rho^{-1})^\nu t)^k}{k!}\cdot e^{-\frac{\rho}{4}|n|}.
\end{align*}
As these estimates hold for arbitrary $k \ge 1$, we can send $k \to \infty$ in the upper bound for $|\hat v(t,n)-\hat w(t,n)|$, and since the limit vanishes, we have
\[\hat v(t,n)\equiv\hat w(t,n),\quad 0\leq t\leq\mathcal L_2, n\in\mathbb Z^\nu.\]
This implies that the spatially quasi-periodic solution to the quasi-periodic Cauchy problem \eqref{sl}--\eqref{dt} is unique. Hence the uniqueness component of Theorem A is proved.
\end{proof}

\section{The General Case: gBBM}

In this section we study the general case, that is, $p\geq2$. In particular we address the remaining cases $p\geq3$. The key point is to propose new indices for the term $u^{p-1}u_x$ and their relations (compared with \cite{DG16JAMS}). Proofs that are similar to the ones in Section~\ref{bm} are omitted; we only give those proofs that need significant new arguments.

To avoid confusion of symbols, set $\hat u(n)=\mathfrak c(n)$ and $\hat u(t,n)=\mathfrak c(t,n)$. The counterparts of $\mathfrak N$, $\mathfrak C$, $\mathfrak I$, $\mathfrak F$, $\sigma$, $\ell$ and $\mathfrak D$ will be denoted by $\mathscr N$, $\mathscr C$, $\mathscr I$, $\mathscr F$, $\alpha$, $\beta$ and $\mathscr D$, respectively. Their definitions will be introduced below.

In the Fourier space, the quasi-periodic Cauchy problem \eqref{sl}--\eqref{dt} is again reduced to a nonlinear infinite system of coupled ODEs,
\begin{align}\label{o3}
(\partial_t\mathfrak c)(t,n)-\lambda(n)\mathfrak c(t,n)=\frac{\lambda(n)}{p}\sum_{\substack{n_1,\cdots,n_p\in\mathbb Z^\nu\\n_1+\cdots+n_p=n}}\prod_{j=1}^p\mathfrak c(t,n_j),\quad \forall n\in\mathbb Z^\nu
\end{align}
with initial data
\begin{align}\label{i3}
\mathfrak c(0,n)=\mathfrak c(n),\quad \forall n\in\mathbb Z^\nu.
\end{align}

Obviously $\mathfrak c(t,0)\equiv\mathfrak c(0)$. According to the variation of constants formula, the Cauchy problem \eqref{o3}--\eqref{i3} is equivalent to the following integral equation,
\begin{align*}
\mathfrak c(t,n)=e^{{\lambda(n)t}}\mathfrak c(n)+\frac{\lambda(n)}{p}\int_0^te^{{\lambda(n)}(t-\tau)}\sum_{\substack{n_1,\cdots,n_p\in\mathbb Z^\nu\\n_1+\cdots+n_p=n}}\prod_{j=1}^p\mathfrak c(\tau,n_j) \, {\rm d}\tau,\quad \forall n\in\mathbb Z^\nu\backslash\{0\}.
\end{align*}

Define the Picard sequence $\{\mathfrak c_k(t,n)\}$ to approximate $\mathfrak c(t,n)$ by letting
\begin{align*}
\mathfrak c_{k}(t,n):=
\begin{cases}
e^{{\lambda(n)t}}\mathfrak c(n), &k=0;\\
\mathfrak c_0(t,n)+\frac{\lambda(n)}{p}\int_0^te^{{\lambda(n)}(t-\tau)}\sum_{\substack{n_1,\cdots,n_p\in\mathbb Z^\nu\\n_1+\cdots+n_p=n}}\prod_{j=1}^p\mathfrak c_{k-1}(\tau,n_j){\rm d}\tau,&k\geq1.
\end{cases}
\end{align*}

Next we will use the combinatorial tree form of the Picard sequence to prove its exponential decay property and then to prove that it is a Cauchy sequence.

Set
\begin{align*}
\P^{(k)}&:=
\begin{cases}
\{0,1\},&k=1;\\
\{0\}\cup(\P^{(k-1)})^{p},&k\geq2.
\end{cases}\\
\mathscr N^{(k,\gamma^{(k)})}&:=
\begin{cases}
\mathbb Z^\nu, &0=\gamma^{(k)}\in\P^{(k)}, k\geq1;\\
(\mathbb Z^\nu)^p, &1=\gamma^{(1)}; \\
\prod_{j=1}^p\mathscr N^{(k-1,\gamma_j^{(k-1)})},&(\gamma_j^{(k-1)})_{1\leq j\leq p}=\gamma^{(k)}\in(\P^{(k-1)})^p, k\geq2.
\end{cases}\\
\mathscr C^{(k,\gamma^{(k)})}(n^{(k)})&:=
\begin{cases}
\mathfrak c(n), &0=\gamma^{(k)}\in\P^{(k)}, n=n^{(k)}\in\mathscr N^{(k,0)},k\geq1;\\
\prod_{j=1}^p\mathfrak c(n_j), &1=\gamma^{(1)}\in\P^{(1)}, (n_j)_{1\leq j\leq p}=n^{(1)}\in\mathscr N^{(1,1)}; \\
\prod_{j=1}^p\mathscr C^{(k-1,\gamma_j^{(k-1)})}(n_j^{(k-1)}),&(\gamma_j^{(k-1)})_{1\leq j\leq p}=\gamma^{(k)}\in(\P^{(k-1)})^p,\\
 &(n_j^{(k-1)})_{1\leq j\leq p}=n^{(k)}\in\prod_{j=1}^p\mathscr N^{(k-1,\gamma_j^{(k-1)})},\\
 &k\geq2.
\end{cases}\\
\mathscr I^{(k,\gamma^{(k)})}(t,n^{(k)})&:=
\begin{cases}
e^{\lambda(n)t}, &0=\gamma^{(k)}\in\P^{(k)}, \\
&n=n^{(k)}\in\mathscr N^{(k,0)},\\
&k\geq1;\\
\int_0^te^{\lambda(\mu(n^{(1)})(t-\tau))}\prod_{j=1}^pe^{\lambda(n_j)\tau}{\rm d}\tau, &1=\gamma^{(1)}\in\P^{(1)}, \\
&(n_j)_{1\leq j\leq p}=n^{(1)}\\
&\in\mathscr N^{(1,1)};\\
\int_0^te^{\lambda(\mu(n^{(k)}))(t-\tau)}\prod_{j=1}^p\mathscr I^{(k-1,\gamma_j^{(k-1)})}(\tau,n_j^{(k-1)}){\rm d}\tau,&(\gamma_j^{(k-1)})_{1\leq j\leq p}=\gamma^{(k)}\\
&\in(\P^{(k-1)})^p,\\
 &(n_j^{(k-1)})_{1\leq j\leq p}=n^{(k)}\\
 &\in\prod_{j=1}^p\mathscr N^{(k-1,\gamma_j^{(k-1)})},\\
 &k\geq2.
\end{cases}\\
\mathscr F^{(k,\gamma^{(k)})}(n^{(k)})&:=
\begin{cases}
1, &0=\gamma^{(k)}\in\P^{(k)}, \\
&n=n^{(k)}\in\mathscr N^{(k,0)},\\
&k\geq1;\\
\frac{\lambda(\mu(n^{(1)}))}{p}, &1=\gamma^{(1)}\in\P^{(1)}, \\
&(n_j)_{1\leq j\leq p}=n^{(1)}\in\mathscr N^{(1,1)};\\
\frac{\lambda(\mu(n^{(k)}))}{p}\prod_{j=1}^p\mathscr F^{(k-1,\gamma_j^{(k-1)})}(n_j^{(k-1)}),&(\gamma_j^{(k-1)})_{1\leq j\leq p}=\gamma^{(k)}\in(\P^{(k-1)})^p,\\
 &(n_j^{(k-1)})_{1\leq j\leq p}=n^{(k)}\\
 &\in\prod_{j=1}^p\mathscr N^{(k-1,\gamma_j^{(k-1)})},k\geq2.
\end{cases}
\end{align*}

With the help of these abstract symbols, we have the following:

\begin{lemm}
The Picard sequence $\{\mathfrak c_k(t,n)\}$ can be rewritten as a combinatorial tree, that is,
\begin{align*}
\mathfrak c_k(t,n)=\sum_{\gamma^{(k)}\in\P^{(k)}}\sum_{\substack{n^{(k)}\in\mathscr N^{(k,\gamma^{(k)})}\\\mu(n^{(k)})=n}}\mathscr C^{(k,\gamma^{(k)})}(n^{(k)})\mathscr I^{(k,\gamma^{(k)})}(t,n^{(k)})\mathscr F^{(k,\gamma^{(k)})}(n^{(k)}),\quad\forall k\geq1.
\end{align*}
\end{lemm}

By induction, we can prove the following estimates for $\mathscr C,\mathscr I$ and $\mathscr F$.
\begin{lemm}
For all $k\geq1$,
\begin{align*}
|\mathscr C^{(k,\gamma^{(k)})}(n^{(k)})|&\leq\mathcal A^{\alpha(\gamma^{(k)})}e^{-\rho|n^{(k)}|},\\
|\mathscr I^{(k,\gamma^{(k)})}(t,n^{(k)})|&\leq\frac{t^{\beta(\gamma^{(k)})}}{\mathscr D(\gamma^{(k)})},\\
|\mathscr F^{(k,\gamma^{(k)})}(n^{(k)})|&\leq\frac{1}{p^{\beta(\gamma^{(k)})}}<1,
\end{align*}
where
\begin{align*}
\alpha(\gamma^{(k)})&:=
\begin{cases}
\frac{1}{p-1},&0=\gamma^{(k)}\in\P^{(k)}, k\geq1;\\
\frac{p}{p-1},&1=\gamma^{(1)}\in\P^{(1)};\\
\sum_{j=1}^p\alpha(\gamma_j^{(k-1)}), &(\gamma_j^{(k-1)})_{1\leq j\leq p}=\gamma^{(k)}\in(\P^{(k-1)})^{p},k\geq2,
\end{cases}\\
\beta(\gamma^{(k)})&:=
\begin{cases}
0,&0=\gamma^{(k)}\in\P^{(k)}, k\geq1;\\
1,&1=\gamma^{(1)}\in\P^{(1)};\\
1+\sum_{j=1}^p\beta(\gamma_j^{(k-1)}), &(\gamma_j^{(k-1)})_{1\leq j\leq p}=\gamma^{(k)}\in(\P^{(k-1)})^{p},k\geq2.
\end{cases}
\end{align*}
\end{lemm}

The following lemma contains some observations we will need below:

\begin{lemm}\label{smf}
For all $k\geq1$,
\begin{enumerate}
\item
  $
  \dim_{\mathbb Z^\nu}\mathscr N^{(k,\gamma^{(k)})}=(p-1)\alpha(\gamma^{(k)});
 $
  \item
  $
  \alpha(\gamma^{(k)})=\beta(\gamma^{(k)})+\frac{1}{p-1};
  $
  \item If $0\leq\natural\leq\frac{(p-1)^{p-1}}{p^p}$, then
  $$\blacklozenge_k\triangleq\sum_{\gamma^{(k)}\in\P^{(k)}}\frac{\natural^{\beta(\gamma^{(k)})}}{\mathscr D(\gamma^{(k)})}\leq\frac{p}{p-1},\quad\forall k\geq1.$$
\end{enumerate}
\end{lemm}

\begin{proof}
We first prove the first two identities. It is easy to see that they are true for $0=\gamma^{(k)}\in\P^{(k)}$, $k\geq1$, and $1=\gamma^{(1)}$. Assume that they hold for $1,\cdots,k-1$, where $k\geq2$. For $k\mapsto(\gamma_j^{(k-1)})_{1\leq j\leq p}=\gamma^{(k)}\in\P^{(k)}$, by the definition of $\alpha,\beta$ and $\mathscr N^{(k,\gamma^{(k)})}$, one can derive that
\begin{align*}
\dim_{\mathbb Z^\nu}\mathscr N^{(k,\gamma^{(k)})}&=\dim_{\mathbb Z^\nu}\prod_{j=1}^p\mathscr N^{(k-1,\gamma_j^{(k-1)})}\\
&=\sum_{j=1}^p\dim_{\mathbb Z^\nu}\mathscr N^{(k-1,\gamma_j^{(k-1)})}\\
&=(p-1)\sum_{j=1}^p\alpha(\gamma_{j}^{(k-1)})\\
&=(p-1)\alpha(\gamma^{(k)})
\end{align*}
and
\begin{align*}
\alpha(\gamma^{(k)})&=\sum_{j=1}^p\alpha(\gamma_j^{(k-1)})\\
&=\sum_{j=1}^{p}\left(\beta(\gamma_j^{(k-1)})+\frac{1}{p-1}\right)\\
&=1+\sum_{j=1}^{p}\beta(\gamma_j^{(k-1)})+\frac{1}{p-1}\\
&=\beta(\gamma^{(k)})+\frac{1}{p-1}.
\end{align*}
By induction, it follows that the first two identities hold for all $k\geq1$.

Next we will prove the last inequality. For $k=1$, we have
\begin{align*}
\blacklozenge_1=1+\natural\leq1+\frac{(p-1)^{p-1}}{p^p}\leq\frac{p}{p-1}.
\end{align*}
This shows that it holds for $k=1$. Let $k\geq2$ and assume that it is true for $1,\cdots,k-1$. For $k$, we have
\begin{align*}
\blacklozenge_k\leq1+\natural\prod_{j=1}^p\sum_{\gamma_j^{(k-1)}\in\P^{(k-1)}}\frac{\natural^{\beta(\gamma_j^{(k-1)})}}{\mathscr D(\gamma_j^{(k-1)})}\leq1+\frac{(p-1)^{(p-1)}}{p^p}\cdot\left(\frac{p}{p-1}\right)^p=\frac{p}{p-1}.
\end{align*}
Hence the last inequality is true for all $k\geq1$.  This completes the proof of Lemma \ref{smf}.
\end{proof}

\begin{rema}
Just as in Remark~\ref{r.optimalchoices}, the value $(p-1)^{p-1}/p^p$ as the upper limit for the range of $\natural$ and the value $p/(p-1)$ as the upper bound for $\blacklozenge_k$ in part (3) of Lemma~\ref{smf} cannot be improved. In fact, assuming that $\blacklozenge_k \leq M'$ for some $M^\prime>1$ and $\natural$ from a suitable interval $(0,r']$, the induction argument yields $1+\natural (M^\prime)^p\leq M^\prime$, which in turn gives $0<\natural\leq(1/M^\prime)^{p-1}-(1/M^\prime)^p$. Consider the following auxiliary function defined by letting $w(x)=x^{p-1}-x^p$, where $0<x<1$. Clearly, $w$ takes its maximum at $x = (p-1)/p$. This implies that $M^\prime=p/(p-1)$ and $0<\natural\leq(p-1)^{p-1}/p^p$ are optimal.
\end{rema}

In a similar way as in the proof of Lemma~\ref{expthm}, we can obtain uniform exponential decay for the Picard sequence $\{\mathfrak c_k(t,n)\}$:

\begin{lemm}\label{gp}
Assume that the initial Fourier coefficients $\mathfrak c(n)$ obey \eqref{exd}. With the constants $\mathcal A$ and $\rho$ from \eqref{exd} and the dimension $\nu$ set
\begin{equation}\label{e.newconstant.p}
\mathscr B_p \triangleq \frac{p}{p-1}\mathcal A(6\rho^{-1})^{\nu}
\end{equation}
and
\begin{equation}\label{e.locextime.p}
\mathcal L_p \triangleq \frac{(p-1)^{p-1} \rho^{(p-1)\nu}}{p^{p-1} \mathcal A 6^{(p-1)\nu}}.
\end{equation}
Then, we have
\begin{align*}
\sup_{\substack{t \in [0,\mathcal L_p] \\ k \ge 0}} |\mathfrak c_k(t,n)|\leq \mathscr B_p e^{-\frac{\rho}{2}|n|}
\end{align*}
for every $n \in \mathbb Z^\nu$.
\end{lemm}

Furthermore using the following pattern decomposition,
\begin{align*}
|\prod_{j=1}^\star\blacktriangle_j-\prod_{j=1}^\star\blacksquare_j|
\leq\sum_{j^\prime=1}^\star\prod_{j=1}^{j^\prime-1}|\blacksquare_j|\cdot|\blacktriangle_{j^\prime}-
,\blacksquare_{j^\prime}|\cdot\prod_{j=j^\prime+1}^{\star}|\blacktriangle_j|,
\end{align*}
where
\[\prod_{j=1}^0|\blacksquare_j|:=1\quad\text{and}\quad\prod_{j=\star+1}^\star|\blacktriangle_j|:=1,\]
we can prove the following by induction:

\begin{lemm}
For all $k\geq1$,
\begin{align*}
|\mathfrak c_{k}(t,n)-\mathfrak c_{k-1}(t,n)| & \leq \frac{p^{k-1}\mathscr B_p^{(p-1)k+1}t^k}{(2p)^k\cdot k!}\sum_{\substack{n_1,\cdots,n_{(p-1)k+1}\in\mathbb Z^\nu\\n_1+\cdots+n_{(p-1)k+1}=n}}\prod_{j=1}^{(p-1)k+1}e^{-\frac{\rho}{2}|n_j|} \\
& \leq \frac{\mathscr B_p(12\rho^{-1})^\nu}{p}\cdot\frac{\left(2^{-1}\mathscr B_p^{p-1}(12\rho^{-1})^{(p-1)\nu}t\right)^k}{k!}\cdot e^{-\frac{\rho}{4}|n|}.
\end{align*}
Hence $\{\mathfrak c_k(t,n)\}$ is a Cauchy sequence on $[0,\mathcal L_p]\times\mathbb Z^\nu$.
\end{lemm}

With these estimates in hand, Theorem A for the quasi-periodic Cauchy problem \eqref{sl}--\eqref{dt}, i.e.\ gBBM, follows in the same way as it did in Section~\ref{bm} for BBM. This completes the proof of Theorem A.

\section{Proof of Theorem B}

In this section we generalize the decay condition from exponential to polynomial for the Fourier coefficients of the quasi-periodic initial data. For the sake of convenience and readability, we take the case of $p=2$ (BBM) as an illustration. This generalization works for the general case.

Specifically, the exponential decay condition \eqref{exd} is replaced by the polynomial decay condition \eqref{pod}, where $1=\nu<\mathtt r$ or $2\leq\nu<\frac{\mathtt r}{4}-2$.

From the proof above, we need to re-estimate only $\mathfrak C$.

\begin{lemm}\label{skd}
If the initial Fourier coefficients satisfy the polynomial decay estimate \eqref{pod}, then
\begin{align}\label{hsl}
|\mathfrak C^{(k,\gamma^{(k)})}(n^{(k)})|\leq\mathtt A^{\sigma(\gamma^{(k)})}\prod_{j=1}^{\sigma(\gamma^{(k)})}\left(1+|(n^{(k)})_j|\right)^{-\mathtt r},\quad\forall k\geq1.
\end{align}
\end{lemm}
\begin{proof}
We first prove the following equality: for all $k\geq1$ and $n^{(k)}=(n_j)_{1\leq j\leq\sigma(\gamma^{(k)})}\in\mathfrak N^{(k,\gamma^{(k)})}$,
\begin{align}\label{ghs}
\mathfrak C^{(k,\gamma^{(k)})}(n^{(k)})=\prod_{j=1}^{\sigma(\gamma^{(k)})}c(n_j).
\end{align}

It is not difficult to see that \eqref{ghs} holds for $0=\gamma^{(k)}\in\spadesuit^{(k)}, k\geq1$, and $1=\gamma^{(1)}\in\spadesuit^{(1)}$.

Let $k\geq2$. Assume that it holds for $1,\cdots,k-1$. For $k,(\gamma_1^{(k-1)},\gamma_2^{(k-1)})=\gamma^{(k)}\in(\spadesuit^{(k-1)})^2$ and $(n_1^{(k-1)},n_2^{(k-1)})=n^{(k)}\in\prod_{j=1}^2\mathfrak N^{(k-1,\gamma_j^{(k-1)})}$, where $n_1^{(k-1)}=(n_j)_{1\leq j\leq\sigma(\gamma_1^{(k-1)})}$ and  $n_2^{(k-1)}=(n_{\sigma(\gamma_1^{(k-1)})+j})_{1\leq j\leq\sigma(\gamma_1^{(k-1)})}$, by the definition of $\mathfrak C$, one can derive that
\begin{align*}
\mathfrak C^{(k,\gamma^{(k)})}(n^{(k)}) & = \mathfrak C^{(k-1,\gamma_1^{(k-1)})}(n_1^{(k-1)})\cdot\mathfrak C^{(k-1,\gamma_2^{(k-1)})}(n_2^{(k-1)}) \\
& = \prod_{j=1}^{\sigma(\gamma_1^{(k-1)})}c(n_j)\cdot\prod_{j=1}^{\sigma(\gamma_2^{(k-1)})}c(n_{\sigma(\gamma_1^{(k-1)})+j}) \\
& = \prod_{j=1}^{\sigma(\gamma^{(k)})}c(n_j).
\end{align*}
By induction, \eqref{ghs} holds for all $k\geq1$. It follows from polynomial decay \eqref{pod} for $c$ that
\begin{align*}
|\mathfrak C^{(k,\gamma^{(k)})}(n^{(k)})| & = \prod_{j=1}^{\sigma(\gamma^{(k)})}|c((n^{(k)})_j)| \\
& \leq \prod_{j=1}^{\sigma(\gamma^{(k)})}\mathtt A(1+|(n^{(k)})_j|)^{-\mathtt r} \\
& = \mathtt A^{\sigma(\gamma^{(k)})}\prod_{j=1}^{\sigma(\gamma^{(k)})}(1+|(n^{(k)})_j|)^{-\mathtt r}.
\end{align*}
This completes the proof of Lemma \ref{skd}.
\end{proof}

In addition, we need the following basic statements:

\begin{lemm}
\begin{enumerate}
  \item (Mean value inequality)
  \begin{align}\label{mie}
  \prod_{j=1}^n\mathtt a_j\leq\left(\frac{1}{n}\sum_{j=1}^n\mathtt a_j\right)^n,\quad \mathtt a_j>0,j=1,\cdots,n\in\mathbb N.
  \end{align}
  \item (Bound for the Riemann zeta function on $\mathbb R$)
  \begin{align}\label{br}
  \sum_{n=1}^{\infty}\frac{1}{n^{\mathtt s}}=:\zeta(\mathtt s)\leq1+\frac{1}{\mathtt s-1},\quad \mathtt s>1.
  \end{align}
  \item Set
  \begin{align}\label{mm}
  \sum_{n\in\mathbb Z^\nu}\frac{1}{(1+|n|)^{\mathtt s}}:=\mathscr H(\mathtt s;\nu).
  \end{align}
  If $1=\nu<\mathtt s$, then $\mathscr H(\mathtt s;1)\leq1+2\zeta(\mathtt s)$; if $2\leq \nu<\mathtt s$, then
   \begin{align}\label{hb}
  \mathscr H(\mathtt s;\nu)
  \leq\mathfrak b(\mathtt s;\nu)\triangleq1+\sum_{j_0=1}^{\nu}\left(\begin{matrix}\nu\\j_0\end{matrix}\right)2^{j_0}j_0^{-\mathtt s}\left\{\zeta\left(\frac{\mathtt s}{j_0}\right)\right\}^{j_0}.
  \end{align}
   \item (Generalized Bernoulli inequality)
  \begin{align}\label{gbi}
  \prod_{j=1}^{m}(1+x_j)\geq1+\sum_{j=1}^{m}x_j, \quad x_j>-1, j=1,\cdots,m\in\mathbb N.
  \end{align}
\end{enumerate}
\end{lemm}

\begin{proof}
\begin{enumerate}

\item It follows from the monotonicity $(\uparrow)$ of the function
  \begin{align}
  f(x):=
  \begin{cases}
  \left(\frac{1}{n}\sum_{j=1}^n\mathtt a_j^x\right)^{\frac{1}{x}},&x\neq0;\\
  \left(\prod_{j=1}^{n}\mathtt a_j\right)^{\frac{1}{n}},&x=0.
  \end{cases}
  \end{align}
  that \eqref{mie}, i.e.  $f(0)\leq f(1)$, is true.

\item Let $g(x)={x^{-\mathtt s}}, \mathtt s>1$. Notice that $g$ is monotonically decreasing on $[1,\infty)$. Thus, we have
\begin{align}\label{sm}
  g(n+1)\leq\int_n^{n+1}g(\tau){\rm d}\tau,\quad \forall n\geq1.
\end{align}
For all $N\geq2$, the summation of \eqref{sm} over $n = 1, \ldots, N-1$ yields
\begin{align*}
  \sum_{n=1}^{N}g(n)\leq g(1)+\int_1^Ng(\tau) \, {\rm d}\tau,
\end{align*}
that is,
$$
\sum_{n=1}^{N}\frac{1}{n^{\mathtt s}}\leq 1+\frac{1}{\mathtt s-1},\quad\text{uniformly for}~~N,~~\text{provided that}~\mathtt s>1.
$$
Hence \eqref{br} holds for all $\mathtt s>1$.

\item Set $\wp:=\{0,\cdots,\nu-1,\nu\}$. For every $j_0\in\wp$, define
$$
\mathcal S_{j_0} := \{ n=(n_1,\cdots,n_\nu) \in \mathbb Z^\nu : \text{exactly } \nu-j_0 \text{ components are equal to zero} \}.
$$
Hence we have the following decomposition,
$$
\mathscr H(\mathtt s;\nu) = 1 +\sum_{n\in\bigcup_{j_0\in\wp\backslash\{0\}}\mathcal S_{j_0}} (1+|n|)^{-\mathtt s} 
\triangleq 1 +(\uppercase\expandafter{\romannumeral4})
$$
On the one hand, for all $j_0\in\wp\backslash\{0\}$, we have
\begin{eqnarray*}
\sum_{n\in\mathcal S_{j_0}}(1+|n|)^{-\mathtt s} & = & \left(\begin{matrix}\nu\\j_0\end{matrix}\right)\sum_{\substack{n=(n_1,\cdots,n_{j_0},0,\cdots,0)\in\mathbb Z^\nu\\n_1,\cdots,n_{j_0}\in\mathbb Z\backslash\{0\}}}(1+|n|)^{-\mathtt s} \\
& = & \left(\begin{matrix}\nu\\j_0\end{matrix}\right)\sum_{n_1,\cdots,n_{j_0}\in\mathbb Z\backslash\{0\}}\left(1+\sum_{j=1}^{j_0}|n_j|\right)^{-\mathtt s} \\
& \leq & \left(\begin{matrix}\nu\\j_0\end{matrix}\right)\sum_{n_1,\cdots,n_{j_0}\in\mathbb Z\backslash\{0\}}\left(\sum_{j=1}^{j_0}|n_j|\right)^{-\mathtt s} \\
& \stackrel{\eqref{mie}}{\leq} & \left(\begin{matrix}\nu\\j_0\end{matrix}\right)j_0^{-\mathtt s}\sum_{n_1,\cdots,n_{j_0}\in\mathbb Z\backslash\{0\}}\prod_{j=1}^{j_0}|n_j|^{-\frac{\mathtt s}{j_0}} \\
& = & \left(\begin{matrix}\nu\\j_0\end{matrix}\right)j_0^{-\mathtt s}\prod_{j=1}^{j_0}\sum_{n_j\in\mathbb Z\backslash\{0\}}|n_j|^{-\frac{\mathtt s}{j_0}} \\
& = & \left(\begin{matrix}\nu\\j_0\end{matrix}\right)2^{j_0}j_0^{-\mathtt s}\prod_{j=1}^{j_0}\sum_{n_j=1}^{\infty}n_j^{-\frac{\mathtt s}{j_0}} \\
& = & \left(\begin{matrix}\nu\\j_0\end{matrix}\right)2^{j_0}j_0^{-\mathtt s}\prod_{j=1}^{j_0}\zeta\left(\frac{\mathtt s}{j_0}\right) \\
& = & \left(\begin{matrix}\nu\\j_0\end{matrix}\right)2^{j_0}j_0^{-\mathtt s}\left\{\zeta\left(\frac{\mathtt s}{j_0}\right)\right\}^{j_0}.
\end{eqnarray*}
Hence we have
\begin{align*}
  (\uppercase\expandafter{\romannumeral4})&=\sum_{n\in\bigcup_{j_0\in\wp\backslash\{0\}}\mathcal S_{j_0}}(1+|n|)^{-\mathtt s}\\
  &=\sum_{j_0\in\wp\backslash\{0\}}\sum_{n\in\mathcal S_{j_0}}(1+|n|)^{-\mathtt s}\\
  &=\sum_{j_0=1}^{\nu}\left(\begin{matrix}\nu\\j_0\end{matrix}\right)2^{j_0}j_0^{-\mathtt s}\left\{\zeta\left(\frac{\mathtt s}{j_0}\right)\right\}^{j_0}.
  \end{align*}

Combining these estimates, we arrive at the following inequality,
\begin{align*}
  \mathscr H(\mathtt s;\nu)\leq1+\sum_{j_0=1}^{\nu}\left(\begin{matrix}\nu\\j_0\end{matrix}\right)2^{j_0}j_0^{-\mathtt s}\left\{\zeta\left(\frac{\mathtt s}{j_0}\right)\right\}^{j_0}\triangleq\mathfrak b(\mathtt s;\nu).
\end{align*}
It follows from \eqref{br} that $\mathscr H(\mathtt s;\nu)$ is a bounded positive number for any fixed $\mathtt s$ and $\nu$ for which $2\leq\nu<\mathtt s$.
   \item This is easily obtained by induction.
\end{enumerate}
\end{proof}

In what follows we will prove that the Picard sequence satisfies a uniform polynomial decay estimate (Lemma \ref{shf}) and is fundamental (Lemma \ref{pq}). The case of $1=\nu$ is trivial and we only discuss the case of $2\leq\nu$.

\begin{lemm}\label{shf}
If $0\leq t\leq\frac{1}{2\mathtt A\mathfrak b\left(\frac{\mathtt r}{2};\nu\right)}\triangleq \mathcal L_2^\prime$ and $2\leq\nu<\frac{\mathtt r}{2}$, then
\begin{align}
|c_k(t,n)|
\leq\mathtt B(1+|n|)^{-\frac{\mathtt r}{2}}, \quad\text{where}~~\mathtt B\triangleq2\mathtt A\mathfrak b\left(\frac{\mathtt r}{2};\nu\right).
\end{align}
\end{lemm}
\begin{proof}
We first have
\begin{eqnarray}
\nonumber |c_k(t,n)| & \stackrel{\eqref{ct}}{\leq}&\sum_{\gamma^{(k)}\in\spadesuit^{(k)}}\sum_{\substack{n^{(k)}\in\mathfrak N^{(k,\gamma^{(k)})}\\\mu(n^{(k)})=n}}|\mathfrak C^{(k,\gamma^{(k)})}(n^{(k)})||\mathfrak I^{(k,\gamma^{(k)})}(t,n^{(k)})||\mathfrak F^{(k,\gamma^{(k)})}(n^{(k)})| \\
\label{sp} & \stackrel{\eqref{hsl},\eqref{i},\eqref{f}}{\leq}&\mathtt A\sum_{\gamma^{(k)}\in\spadesuit^{(k)}}\frac{(2^{-1}\mathtt At)^{\ell(\gamma^{(k)})}}{\mathfrak D(\gamma^{(k)})}\sum_{\substack{n^{(k)}\in\mathfrak N^{(k,\gamma^{(k)})}\\\mu(n^{(k)})=n}}\prod_{j=1}^{\sigma(\gamma^{(k)})}(1+|(n^{(k)})_j|)^{-\mathtt r}.
\end{eqnarray}
The main difference, compared to the proof in the exponential decay case, is to deal with the term
\begin{align}
&\sum_{\substack{n^{(k)}\in\mathfrak N^{(k,\gamma^{(k)})}\\\mu(n^{(k)})=n}}\prod_{j=1}^{\sigma(\gamma^{(k)})}(1+|(n^{(k)})_j|)^{-\mathtt r}\nonumber\\
\label{slr}=&\sum_{\substack{n^{(k)}\in\mathfrak N^{(k,\gamma^{(k)})}\\\mu(n^{(k)})=n}}\prod_{j=1}^{\sigma(\gamma^{(k)})}(1+|(n^{(k)})_j|)^{-\frac{\mathtt r}{2}}
\cdot\underbrace{\prod_{j=1}^{\sigma(\gamma^{(k)})}(1+|(n^{(k)})_j|)^{-\frac{\mathtt r}{2}}}_{\boxdot}.
\end{align}
It follows from the generalized Bernoulli inequality \eqref{gbi} and $\mu(n^{(k)})=n$ that
\begin{eqnarray}
\nonumber\boxdot&=&\prod_{j=1}^{\sigma(\gamma^{(k)})}(1+|(n^{(k)})_j|)^{-\frac{\mathtt r}{2}}\\
\nonumber&=&\left(\prod_{j=1}^{\sigma(\gamma^{(k)})}(1+|(n^{(k)})_j|)\right)^{-\frac{\mathtt r}{2}}\\
\nonumber&\stackrel{\eqref{gbi}}{\leq}&\left(1+\sum_{j=1}^{\sigma(\gamma^{(k)})}|(n^{(k)})_j|\right)^{-\frac{\mathtt r}{2}}\\
\nonumber&=&(1+|n^{(k)}|)^{-\frac{\mathtt r}{2}}\\
\label{as}&\leq&(1+|n|)^{-\frac{\mathtt r}{2}}.
\end{eqnarray}
Inserting \eqref{as} into \eqref{slr} yields
\begin{align}
\nonumber \sum_{\substack{n^{(k)}\in\mathfrak N^{(k,\gamma^{(k)})}\\\mu(n^{(k)})=n}}\prod_{j=1}^{\sigma(\gamma^{(k)})}(1+|(n^{(k)})_j|)^{-\mathtt r} & \leq \sum_{n^{(k)}\in\mathfrak N^{(k,\gamma^{(k)})}}\prod_{j=1}^{\sigma(\gamma^{(k)})}(1+|(n^{(k)})_j|)^{-\frac{\mathtt r}{2}}\cdot(1+|n|)^{-\frac{\mathtt r}{2}} \\
\nonumber & \leq \sum_{n_1,\cdots,n_{\sigma(\gamma^{(k)})}\in\mathbb Z^\nu}\prod_{j=1}^{\sigma(\gamma^{(k)})}(1+|n_j|)^{-\frac{\mathtt r}{2}}\cdot(1+|n|)^{-\frac{\mathtt r}{2}} \\
\nonumber & = \prod_{j=1}^{\sigma(\gamma^{(k)})}\underbrace{\sum_{n_j\in\mathbb Z^\nu}(1+|n_j|)^{-\frac{\mathtt r}{2}}}_{=\mathscr H\left(\frac{\mathtt r}{2};\nu\right)\leq\mathfrak b\left(\frac{\mathtt r}{2};\nu\right)~~\text{by}~~\eqref{hb}}\cdot(1+|n|)^{-\frac{\mathtt r}{2}} \\
\label{sld} & \leq \left\{\mathfrak b\left(\frac{\mathtt r}{2};\nu\right)\right\}^{\sigma(\gamma^{(k)})}(1+|n|)^{-\frac{\mathtt r}{2}}.
\end{align}
Inserting \eqref{sld} into \eqref{sp}, we obtain
\begin{align*}
|c_k(t,n)|
\leq\mathtt A\mathfrak b\left(\frac{\mathtt r}{2};\nu\right)\sum_{\gamma^{(k)}\in\spadesuit^{(k)}}\frac{\left\{2^{-1}\mathtt A\mathfrak b\left(\frac{\mathtt r}{2};\nu\right)t\right\}^{\ell(\gamma^{(k)})}}{\mathfrak D(\gamma^{(k)})}\cdot(1+|n|)^{-\frac{\mathtt r}{2}}.
\end{align*}
It follows from \eqref{4} that
\[
|c_k(t,n)|
\leq\mathtt B(1+|n|)^{-\frac{\mathtt r}{2}}, \quad\text{where}~~\mathtt B\triangleq2\mathtt A\mathfrak b\left(\frac{\mathtt r}{2};\nu\right),
\]
provided that
\[
0\leq t \leq \mathcal L_2^\prime = \frac{1}{2\mathtt A\mathfrak b\left(\frac{\mathtt r}{2};\nu\right)}.
\]
This completes the proof of Lemma~\ref{shf}.
\end{proof}

\begin{lemm}\label{pq}
If $2\leq\nu<\frac{\mathtt r}{4}$, then for all $k\geq1$,
\begin{align}
\label{spr} |c_k(t,n)-c_{k-1}(t,n)| & \leq \frac{2^{k-1}\mathtt B^{k+1}t^k}{4^k\cdot k!}\sum_{\substack{n_1,\cdots,n_{k+1}\in\mathbb Z^\nu\\n_1+\cdots+n_{k+1}=n}}\left\{\prod_{j=1}^{k+1}(1+|n_j|)\right\}^{-\frac{\mathtt r}{2}} \\
& \leq \frac{\mathtt B\mathfrak b\left(\frac{\mathtt r}{4};\nu\right)}{2}\cdot\frac{\left\{2^{-1}\mathtt B\mathfrak b\left(\frac{\mathtt r}{4};\nu\right)t\right\}^{k}}{k!}\cdot\left\{1+|n|\right\}^{-\frac{\mathtt r}{4}}.
\end{align}
This implies that $\{c_k(t,n)\}$ is a Cauchy sequence on $(t,n)\in [ 0,\mathcal L_2^\prime ]\times\mathbb Z^\nu$.
\end{lemm}

\begin{proof}
For $k=1$ one has
\begin{align*}
|c_1(t,n)-c_0(t,n)| & \leq \frac{1}{4}\int_0^t\sum_{\substack{n_1,n_2\in\mathbb Z^\nu\\n_1+n_2=n}}\prod_{j=1}^2|c_0(s,n_j)| \, {\rm d}\tau \\
& = \frac{1}{4}\int_0^t\sum_{\substack{n_1,n_2\in\mathbb Z^\nu\\n_1+n_2=n}}\prod_{j=1}^2|c(n_j)| \, {\rm d}\tau \\
& \leq \frac{\mathtt B^2t}{4}\sum_{\substack{n_1,n_2\in\mathbb Z^\nu\\n_1+n_2=n}}\left\{\prod_{j=1}^2(1+|n_j|)\right\}^{-\frac{\mathtt r}{2}}.
\end{align*}
This shows that \eqref{spr} is true for $k=1$. Let $k\geq2$ and suppose that it holds for $1,\cdots,k-1$. For $k$, one can derive that
\begin{align*}
|c_{k}(t,n)-c_{k-1}(t,n)| & \leq \frac{1}{4}\int_0^t\sum_{\substack{n_1,n_2\in\mathbb Z^\nu\\n_1+n_2=n}}|c_{k-1}(\tau,n_1)||c_{k-1}(\tau,n_2)-c_{k-2}(\tau,n_2)| \, {\rm d}\tau \triangleq(\uppercase\expandafter{\romannumeral1^\prime}) \\
& \quad + \frac{1}{4}\int_0^t\sum_{\substack{n_1,n_2\in\mathbb Z^\nu\\n_1+n_2=n}}|c_{k-1}(\tau,n_1)-c_{k-2}(\tau,n_1)||c_{k-2}(\tau,n_2)| \, {\rm d}\tau\triangleq(\uppercase\expandafter{\romannumeral2^\prime}),
\end{align*}
where
\begin{align*}
(\uppercase\expandafter{\romannumeral1^\prime}) & \leq \frac{1}{4}\int_0^t\sum_{\substack{n_1,n_2\in\mathbb Z^\nu\\n_1+n_2=n}}\mathtt B(1+|n_1|)^{-\frac{\mathtt r}{2}}\cdot\frac{2^{k-2}\mathtt B^{k}\tau^{k-1}}{4^{k-1}\cdot(k-1)!}\sum_{\substack{m_1,\cdots,m_k\in\mathbb Z^\nu\\m_1+\cdots+m_k=n_2}}\left\{\prod_{j=1}^k(1+|m_j|)\right\}^{-\frac{\mathtt r}{2}} \, {\rm d}\tau \\
& = \frac{2^{k-2}\mathtt B^{k+1}t^k}{4^k\cdot k!}\sum_{\substack{n_1,n_2\in\mathbb Z^\nu\\n_1+n_2=n}}\sum_{\substack{m_1,\cdots,m_k\in\mathbb Z^\nu\\m_1+\cdots+m_k=n_2}}(1+|n_1|)^{-\frac{\mathtt r}{2}}\left\{\prod_{j=1}^k(1+|m_j|)\right\}^{-\frac{\mathtt r}{2}} \\
& = \frac{2^{k-2}\mathtt B^{k+1}t^k}{4^k\cdot k!}\sum_{\substack{n_1,\cdots,n_{k+1}\in\mathbb Z^\nu\\n_1+\cdots+n_{k+1}=n}}\left\{\prod_{j=1}^{k+1}(1+|n_j|)\right\}^{-\frac{\mathtt r}{2}},
\end{align*}
and analogously
\begin{align*}
(\uppercase\expandafter{\romannumeral2^\prime})
\leq&\frac{2^{k-2}\mathtt B^{k+1}t^k}{4^k\cdot k!}\sum_{\substack{n_1,\cdots,n_{k+1}\in\mathbb Z^\nu\\n_1+\cdots+n_{k+1}=n}}\left\{\prod_{j=1}^{k+1}(1+|n_j|)\right\}^{-\frac{\mathtt r}{2}}.
\end{align*}
Hence we have
\begin{align*}
|c_k(t,n)-c_{k-1}(t,n)| & \leq (\uppercase\expandafter{\romannumeral1^\prime})+(\uppercase\expandafter{\romannumeral2^\prime}) \\
& \leq \frac{2^{k-1}\mathtt B^{k+1}t^k}{4^k\cdot k!}\sum_{\substack{n_1,\cdots,n_{k+1}\in\mathbb Z^\nu\\n_1+\cdots+n_{k+1}=n}}\left\{\prod_{j=1}^{k+1}(1+|n_j|)\right\}^{-\frac{\mathtt r}{2}} \\
& \leq \frac{2^{k-1}\mathtt B^{k+1}t^k}{4^k\cdot k!}\sum_{\substack{n_1,\cdots,n_{k+1}\in\mathbb Z^\nu\\n_1+\cdots+n_{k+1}=n}}\left\{\prod_{j=1}^{k+1}(1+|n_j|)\right\}^{-\frac{\mathtt r}{4}}\cdot
\left\{\prod_{j=1}^{k+1}(1+|n_j|)\right\}^{-\frac{\mathtt r}{4}} \\
& \leq \frac{2^{k-1}\mathtt B^{k+1}t^k}{4^k\cdot k!}\sum_{\substack{n_1,\cdots,n_{k+1}\in\mathbb Z^\nu}}\left\{\prod_{j=1}^{k+1}(1+|n_j|)\right\}^{-\frac{\mathtt r}{4}}\cdot
\left\{1+\sum_{j=1}^{k+1}|n_j|)\right\}^{-\frac{\mathtt r}{4}} \\
& \leq \frac{2^{k-1}\mathtt B^{k+1}t^k}{4^k\cdot k!}\prod_{j=1}^{k+1}\underbrace{\sum_{n_j\in\mathbb Z^\nu}\left\{(1+|n_j|)\right\}^{-\frac{\mathtt r}{4}}}_{_{=\mathscr H\left(\frac{\mathtt r}{4};\nu\right)\leq\mathfrak b\left(\frac{\mathtt r}{4};\nu\right)~~\text{by}~~\eqref{hb}}}\cdot
\left\{1+|n|\right\}^{-\frac{\mathtt r}{4}} \\
& = \frac{\mathtt B\mathfrak b\left(\frac{\mathtt r}{4};\nu\right)}{2}\cdot\frac{\left\{2^{-1}\mathtt B\mathfrak b\left(\frac{\mathtt r}{4};\nu\right)t\right\}^{k}}{k!}\cdot\left\{1+|n|\right\}^{-\frac{\mathtt r}{4}}.
\end{align*}
This completes the proof of Lemma \ref{pq}.
\end{proof}

We are now in a position to prove our second main result, Theorem B.

\begin{proof}[Proof of Theorem~B]
The existence proof is similar to the case of exponential decay. The uniqueness proof is analogous to proving that the Picard sequence is a Cauchy sequence. We mainly give a convergence analysis to show that the solution we construct is in the classical sense.
In fact, for $\#=0,1,2$, one can derive that
\begin{eqnarray*}
\sum_{n\in\mathbb Z^\nu}|n|^{\#}|c(t,n)| & {\lesssim} & \sum_{n\in\mathbb Z^\nu}|n|^{\#}(1+|n|)^{-\frac{\mathtt r}{2}} \\
& \leq & \sum_{n\in\mathbb Z^\nu}(1+|n|)^{\#-\frac{\mathtt r}{2}}\\
&=&\mathscr H\left(\frac{\mathtt r}{2}-\#;\nu\right)
\end{eqnarray*}
and
\begin{eqnarray*}
\sum_{n\in\mathbb Z^\nu}|n|^{\#}\sum_{\substack{n_1,n_2\in\mathbb Z^\nu\\n_1+n_2=n}}\prod_{j=1}^2|c(t,n_j)| & \lesssim & \sum_{n\in\mathbb Z^\nu}|n|^{\#}\sum_{\substack{n_1,n_2\in\mathbb Z^\nu\\n_1+n_2=n}}\prod_{j=1}^2(1+|n_j|)^{-\frac{\mathtt r}{2}} \\
& \stackrel{\eqref{gbi}}{\leq} & \sum_{n\in\mathbb Z^\nu}\sum_{\substack{n_1,n_2\in\mathbb Z^\nu\\n_1+n_2=n}}(|n_1|+|n_2|)^{\#}\prod_{j=1}^2(1+|n_j|)^{-\frac{\mathtt r}{4}}\cdot(1+|n|)^{-\frac{\mathtt r}{4}} \\
& \leq & \sum_{n\in\mathbb Z^\nu}(1+|n|)^{-\frac{\mathtt r}{4}}\sum_{\substack{n_1,n_2\in\mathbb Z^\nu\\n_1+n_2=n}}\left((1+|n_1|)(1+|n_2|)\right)^{\#}\prod_{j=1}^2(1+|n_j|)^{-\frac{\mathtt r}{4}} \\
& = & \sum_{n\in\mathbb Z^\nu}(1+|n|)^{-\frac{\mathtt r}{4}}\sum_{\substack{n_1,n_2\in\mathbb Z^\nu\\n_1+n_2=n}}\prod_{j=1}^2(1+|n_j|)^{\#-\frac{\mathtt r}{4}} \\
& \leq & \sum_{n\in\mathbb Z^\nu}(1+|n|)^{-\frac{\mathtt r}{4}}\sum_{\substack{n_1,n_2\in\mathbb Z^\nu}}\prod_{j=1}^2(1+|n_j|)^{\#-\frac{\mathtt r}{4}} \\
& \leq & \sum_{n\in\mathbb Z^\nu}(1+|n|)^{-\frac{\mathtt r}{4}}\prod_{j=1}^2\sum_{\substack{n_j\in\mathbb Z^\nu}}(1+|n_j|)^{\#-\frac{\mathtt r}{4}} \\
& = & \mathscr H\left(\frac{\mathtt r}{4};\nu\right)\left\{\mathscr H\left(\frac{\mathtt r}{4}-\#;\nu\right)\right\}^2.
\end{eqnarray*}
It follows from \eqref{mm} that the convergence needed can be guaranteed if
\[\quad2\leq\nu<\frac{\mathtt r}{4}-2
.\]
This completes the proof of Theorem B.
\end{proof}

\bibliographystyle{alpha}
\bibliography{bbm}

\end{document}